\documentclass[a4paper,oneside,11pt,reqno]{article}
\usepackage{amssymb}%

\usepackage[T1]{fontenc}
\usepackage[english]{babel}
\usepackage{amsmath,amssymb,amsthm}
\usepackage[margin=3.7cm]{geometry}
\usepackage{esint}
\usepackage{eucal}
\usepackage[font=small]{caption}
\usepackage[caption=false]{subfig}
\usepackage{graphicx}
\usepackage{listings}
\usepackage{tikz}
\graphicspath{ {./figures/}}

\usepackage{fancyhdr}
 \newtheorem{theorem}{Theorem}
 \newtheorem{lemma}[theorem]{Lemma}
 \newtheorem{proposition}[theorem]{Proposition}
 \newtheorem{corollary}[theorem]{Corollary}

 \newtheorem{remark}[theorem]{Remark}
\theoremstyle{definition}

 \numberwithin{theorem}{section}
 \numberwithin{equation}{section}

\author{Guanglian Li\thanks{Institut f\"ur Numerische Simulation,
            Universit\"at Bonn,
            Wegelerstra{\ss}e 6, D-53115 Bonn, Germany;
             li@ins.uni-bonn.de, peterseim@ins.uni-bonn.de, schedensack@ins.uni-bonn.de
             }\;\,\thanks{The authors gratefully acknowledge support by the Hausdorff
                Center for Mathematics Bonn.}
         \;and Daniel Peterseim\footnotemark[1]\;\,\footnotemark[2]\;\,\thanks{D.~Peterseim 
               gratefully acknowledges support by  Deutsche Forschungsgemeinschaft in the
                Priority Program 1748 "Reliable simulation techniques in solid mechanics:
                Development of non-standard discretization methods, mechanical
                and mathematical analysis" under the project
                "Adaptive isogeometric modeling of propagating strong
                discontinuities in heterogeneous materials".}
          \;and Mira Schedensack\footnotemark[1]\;\,\footnotemark[2]
}
\title{Error analysis of a variational multiscale stabilization for 
convection-dominated diffusion equations in 2d}
 \date{\today}

\newcommand{\prnt}[1]{\left( #1 \right)}

\newcommand{\norm}[1]{\left\|#1\right\|}
\newcommand{\normHsemi}[2]{\left|#1\right|_{H^{1}\prnt{#2}}}

\newcommand{\normI}[2]{\norm{#1}_{L^{\infty}\prnt{#2}}}
\newcommand{\normL}[2]{\norm{#1}_{L^2\prnt{#2}}}
\usepackage{soul}
\setstcolor{red}

\newcommand{\tri}{\mathcal{T}}
\newcommand{\corrector}{\mathcal{C}}

\newcommand{\CI}{C_{I_H}\hspace{-0.5ex}\big(\tfrac{H}{h}\big)}
\newcommand{\Col}[1][\ell]{C_{\mathrm{ol},#1}(\epsilon)}
\newcommand{\Peclet}[1]{\mathrm{Pe}_{H,b,\epsilon}\,}

\begin{document}
\maketitle
\begin{abstract}
We formulate a stabilized quasi-optimal Petrov-Galerkin method for 
singularly perturbed convection-diffusion problems based on the variational 
multiscale method. 
The stabilization is of Petrov-Galerkin type with a standard finite element 
trial space and a problem-dependent test space based on pre-computed 
fine-scale correctors. The exponential decay of these correctors and their 
localisation to local patch problems, which depend on the direction of 
the velocity field and the singular perturbation parameter, is rigorously justified. 
Under moderate assumptions, this stabilization guarantees stability and 
quasi-optimal rate of convergence for arbitrary mesh P\'eclet numbers 
on fairly coarse meshes at the cost of additional inter-element communication.
\end{abstract}

\section{Introduction}

Given a domain $\Omega\subset\mathbb{R}^2$, a singular perturbation parameter 
$0<\epsilon\leq 1$, a velocity field $b\in (L^\infty(\Omega))^2$ and 
some force $f\in H^{-1}(\Omega)$, the convection-diffusion equation 
seeks $u\in V:=H^1_0(\Omega)$ such that
\begin{equation}
\begin{aligned} \label{eq:original}
-\epsilon\Delta u+b\cdot\nabla u&=f \quad \text{in } \Omega,\\
u&=0 \quad \text{on } \partial\Omega.
\end{aligned}
\end{equation}
We assume that the velocity field $b$ is incompressible, i.e., $\nabla\cdot b=0$. 
The focus of this paper is on the convection-dominated regime 
with large P\'eclet number $\mathrm{Pe}=\normI{b}{\Omega}/\epsilon$.

For reasonable small P\'eclet numbers, classical Galerkin finite element 
methods (FEMs) perform well. However, if the P\'eclet number 
increases, then steep gradients of $u$ occur and boundary layers appear, 
which require a much finer mesh to capture the characteristic width of those boundary layers. 
Consequently, local corrections are needed at those layers and a numerical method in which the smooth solution 
regions are not polluted by those layers is desirable. 
The thickness of the parabolic layer is $\mathcal{O}(\sqrt{\epsilon})$ and 
$\mathcal{O}(\epsilon)$ for the exponential layer, which have to be resolved 
for a stable approximation with a standard Galerkin FEM. 
Furthermore, it holds that 
$\normHsemi{u}{\Omega^*}=\mathcal{O}(\epsilon^{-\frac{1}{4}})$ and 
$\normHsemi{u}{\Omega^o}=\mathcal{O}(\epsilon^{-\frac{1}{2}})$ with 
small neighbourhoods $\Omega^*$ and $\Omega^o$  of the parabolic 
and the exponential boundary layer, respectively \cite{roos1996numerical, John2009}.

Numerous numerical methods have been proposed in the past few decades aiming at 
solving the convection dominated problem \eqref{eq:original} efficiently and accurately.  
Upwinding methods for stabilization of the exponential boundary layers 
combined with refinement near the parabolic boundary layers are formulated. 
Among them are streamline upwind/Petrov-Galerkin method (SUPG) or 
Galerkin least squares method (GLS) \cite{FRANCA1992253,Christiansen2016}, 
hp finite element methods \cite{melenk1997robust,Melenk2002}, 
discontinuous Petrov-Galerkin methods (DPG) \cite{Demkowicz2012396}, 
residual-free bubble approaches (RFB) \cite{Brezzi2000,CangianiSueli2005a,CangianiSueli2005b}, 
methods with an additional non-linear diffusion \cite{Barrenechea2016},
methods with stabilization by local orthogonal sub-scales \cite{Codina2000}
and hybridizable discontinuous Galerkin (HDG) methods \cite{Qiu2015}.
Among the multiscale methods are
variational multiscale methods (VMS) \cite{Hughes&Sangalli2007,LarsonMalqvist2009}, 
multiscale finite element methods (MsFEM) \cite{Park&Hou2004, Efendiev15convection}, 
multiscale hybrid-mixed methods (MHM) \cite{HarderParedesValentin2015} and
local orthogonal decomposition methods (LOD) \cite{Elfverson2015}. Specifically, 
the residual-based stabilization methods (SUPG, GLS and RFB) incorporate global stability properties 
into high accuracy in local regions away from boundary layers. We refer to \cite{roos1996numerical} 
for an overview of robust numerical methods for singular perturbed problems.
In this paper, our focus is on the construction and the error analysis of a 
stable and accurate LOD method based 
on \cite{Hughes&Sangalli2007,Peterseim:2014, Mlqvist.Peterseim:2011}.

VMS was designed for solving multiscale problems by embedding fine-scale information into 
the coarse-scale framework. 
Essentially, the efficiency and accuracy rely on the construction of a problem-dependent 
stable projector from a larger fine space onto a relatively much smaller coarse space. 
Our motivation for this paper is originated from
\cite{Hughes&Sangalli2007}, where the authors derived an explicit formula for the one-dimensional 
fine-scale Green's function arising in VMS. 
The smaller the support of the fine-scale Green's function, the more favorable the localized 
method (e.g., \cite{Mlqvist.Peterseim:2011, Peterseim:2014}) in solving~\eqref{eq:original}. 
In particular, the authors compared the fine-scale Green's functions derived by 
the $L^2$-projector with that derived by the $H_{0}^{1}$-projector, and concluded that the latter 
outweighed the former in the one-dimensional case. In addition, examples were 
shown for the two-dimensional 
case that the $H_{0}^{1}$-projector would exceed the $L^2$-projector as well. 
There is a recent work \cite{Elfverson2015} on the convection-diffusion 
problem employing the $L^2$-projector in the framework of VMS and LOD. 
The author shows convergence of the localized method and tests the method using 
$H_{0}^1$-projector and claim that the superiority of $H_{0}^{1}$-projector over the 
$L^2$-projector is not valid for the two-dimensional case.

In the one-dimensional case, the $H^1_0$ projection equals the nodal
interpolation. Therefore, another possible generalization of the 1d case to 
higher dimensions is to use nodal interpolation in the VMS. This approach 
was previously utilised in~\cite{LarsonMalqvist2009} and seems to work 
better than averaging type operators.
In this paper, we show that a VMS based on the nodal interpolation operator coupled with 
a Petrov-Galerkin method is stable and locally quasi-optimal 
for the convection-dominated problem \eqref{eq:original} with 
no spurious oscillations and no smearing.
As for other elliptic PDEs the ideal VMS is turned into a practical 
method by localizing the support of the VMS basis functions \cite{Peterseim:2015}.
Inspired by the numerical results of the fine-scale Green's functions displayed 
in \cite{Hughes&Sangalli2007} and the proof in our paper as well, 
a $b$-biased local region is proposed as the numerical domain for approximating the ideal method. 
The convergence of this localization is proved under the assumption that the local 
region is sufficiently large.

The remainder of this paper is organized as follows. In Section~\ref{s:modelproblem}, a detailed description 
of the problem considered in this paper is shown. 
In Section~\ref{s:idealMethod}, we propose a new VMS method based on the nodal 
interpolation and denote it as the ideal method. 
Its stability and local quasi-optimality are displayed. 
In Section~\ref{s:exponentialdecay}, we estimate the error of the global 
correctors outside a certain local 
patch and show an exponential decay of the error with respect to the size of the local patch. 
Inspired by the results in Section~\ref{s:exponentialdecay}, we formulate the localization algorithm in 
Section~\ref{s:lod} for the ideal method proposed in Section~\ref{s:idealMethod}, and display the 
stability of this algorithm as well as the convergence. 
A numerical experiment is provided in Section~\ref{s:numerics} for the validation of 
our method and we end this paper with conclusions in Section~\ref{s:conclusions}.



\section{Model problem and standard finite elements}\label{s:modelproblem}

We assume that the parameter $\epsilon\leq 1$ and 
$b\in L^{\infty}(\Omega;\mathbb{R}^2)$ is a divergence-free vector field 
and we define the bilinear 
form $a$ on $V\times V$ associated to~\eqref{eq:original} by
\begin{align}
a(u,v)=\epsilon\int_{\Omega}\nabla u\cdot \nabla v\,dx
    +\int_{\Omega}(b\cdot\nabla u)\,v\,dx
  \qquad\text{for all }u,v\in V.
\end{align}
Since $\nabla\cdot b=0$, an integration by parts implies that 
the bilinear form $a$ is $V$-elliptic, i.e., 
\begin{equation}\label{eqn:v-elliptic}
\begin{aligned}
  a(v,v)&= \epsilon \normHsemi{v}{\Omega}^2
   \qquad\text{ for all }v\in V.
\end{aligned}
\end{equation}
Furthermore, a Poincar\'e-Friedrichs inequality leads to the 
existence of some $C(\Omega,b)$ that may depend on (the diameter of) the 
domain $\Omega$ and the $L^\infty$-norm of $b$ such that $a$ is continuous, i.e.,
for all $u,v\in V$ it holds that
\begin{equation}\label{eqn:continuitya}
\begin{aligned}
a(u,v)&\leq \epsilon \normHsemi{u}{\Omega} \normHsemi{v}{\Omega}
    + \|b\|_{L^\infty(\Omega)} \normHsemi{u}{\Omega} \normL{v}{\Omega}\\
  &\leq C(\Omega,b) \normHsemi{u}{\Omega}\normHsemi{v}{\Omega}.
\end{aligned}
\end{equation}
Here, we used that $\epsilon\leq 1$.
Throughout this paper, $A\lesssim B$ abbreviate that there exists a constant $C>0$ independent 
of $\epsilon$, $h$ and $H$ ($h$ and $H$ will be defined later), such that 
$A\leq C B$, and let $A\gtrsim B$ be defined as $B\lesssim A$ 
and $A\approx B$ abbreviates $A\lesssim B\lesssim A$.
We assume that 
$\normI{b}{\Omega}\approx 1$. 
Let $\langle\bullet,\bullet\rangle_{H^{-1}(\Omega)\times H^1_0(\Omega)}$ 
denote the dual pairing of $H^{-1}(\Omega)$ and $H^1_0(\Omega)$.

We consider the variational form of~\eqref{eq:original}:
\begin{equation}\label{eqn:variationalForm}
\left\{
\begin{aligned}
&\text{find $u\in V$ such that for all }v\in V\\
&a(u,v)=\langle f,v\rangle_{H^{-1}(\Omega)\times H^1_0(\Omega)}.
\end{aligned}
\right.
\end{equation}
By virtue of the $V$-ellipticity and $V$-continuity of $a$ 
from~\eqref{eqn:v-elliptic} and~\eqref{eqn:continuitya} and 
the Lax-Milgram lemma, problem~\eqref{eqn:variationalForm} has a unique solution in $V$.

Let $\tri_h$ be a shape-regular triangulation of the domain $\Omega$,
where $h$ represents the minimal diameter of all triangles in $\tri_h$. 
Given a triangulation $\tri$, let
\begin{align*}
{\mathcal{P}}_1(\tri)  &:= \{ {v} \in C^{0}(\Omega)\mid {v}|_K \in P^1(K)
\mbox{ for all } K \in \tri \},
\end{align*}
denote the space of piecewise linear finite elements and define 
$V_h:={\mathcal{P}}_1(\tri_h)\cap V$.

Let $u_h\in V_h$ denote the reference solution, which is defined as
the Galerkin approximation that satisfies
\begin{equation}\label{eqn:fine-scale}
   a(u_h,v_h)=\langle f,v_h\rangle_{H^{-1}(\Omega)\times H^1_0(\Omega)} 
      \qquad \text{for all }v_h\in V_h.
\end{equation}
Taking advantage of the ellipticity and continuity of $a$ 
from~\eqref{eqn:v-elliptic} and~\eqref{eqn:continuitya} on 
$V\times V\supset V_h\times V_h$, the Lax-Milgram lemma implies that 
the fine-scale solution $u_h$ 
of \eqref{eqn:fine-scale} exists and is unique on $V_h$. 

We assume that $\epsilon\ll 1$ is a small parameter and that $\tri_h$
resolves $\epsilon$ in the sense that $u_h$ is a good approximation of $u$,
e.g., if 
\begin{align}\label{eqn:meshsizeh}
  h_\mathrm{max}\|b\|_{L^\infty(\Omega)}/\epsilon\lesssim 1
\end{align}
with the maximal mesh-size $h_\mathrm{max}$ of $\tri_h$.
It holds that 
\begin{align*}
 \normHsemi{u-u_h}{\Omega}
   \lesssim\left(1+\frac{h_\mathrm{max}\|b\|_{L^\infty(\Omega)}}{\epsilon}\right) 
       \inf_{v_h\in V_h} \normHsemi{u-v_h}{\Omega}.
\end{align*}
If, in addition, the solution $u$ of~\eqref{eqn:variationalForm} 
satisfies $u\in H^2(\Omega)$,
standard interpolation estimates lead to
\begin{align*}
 \normHsemi{u-u_h}{\Omega}
     \lesssim h_\mathrm{max}\left(1+\frac{h_\mathrm{max}\|b\|_{L^\infty(\Omega)}}{\epsilon}\right) 
          \|u\|_{H^2(\Omega)},
\end{align*}
with a hidden constant independent of $\epsilon$. 
Note, however, that $\|u\|_{H^2(\Omega)}$ depends on $\epsilon$.

\section{The ideal method}\label{s:idealMethod}

In this section we introduce a variational multiscale method based 
on the nodal interpolation, which 
yields a locally best-approximation of the reference solution $u_h\in V_h$
from~\eqref{eqn:fine-scale} and which is computed on a feasible coarse 
underlying mesh $\tri_H$.
We assume that $\tri_H$ is a regular quasi-uniform triangulation of 
the domain $\Omega$ with maximal mesh-size $H$, such that $\tri_h$ is a refinement of 
$\tri_H$. 
Let $\mathcal{N}_H$ denote the nodes in $\tri_H$ and $\mathrm{mid}_K$ the 
baricenter for each coarse element $K\in \tri_H$. 
The maximal mesh-size $H$ of $\tri_H$ represents a computationally feasible 
scale that is typically much larger than $\epsilon$. Altogether, the 
target regime is then 
\begin{align*}
  0<h< \epsilon \ll  H\lesssim 1.
\end{align*}

Define $V_H={\mathcal{P}}_1(\tri_H)\cap V$ and let
$I_H: V_h\rightarrow V_H$ denote the nodal interpolation. Note that 
$I_H$ acts only on finite element functions and is, hence, well defined. 
It holds,
\begin{align}\label{eqn:interpolation}
&H^{-1}\normL{v-I_Hv}{T}+\normHsemi{I_Hv}{T}
\leq \CI \normHsemi{v}{T}.
\end{align}
Indeed, we have \cite{Yserentant86}
\begin{align*}
  \CI\lesssim \begin{cases}
                1  & \quad\text{ in one dimension},\\
                \log{\frac{H}{h}} & \quad\text{ in two dimensions},\\
                \frac{H}{h} & \quad\text{ in three dimensions}.
              \end{cases}
\end{align*}

\begin{figure}
\includegraphics[width=.5\textwidth]{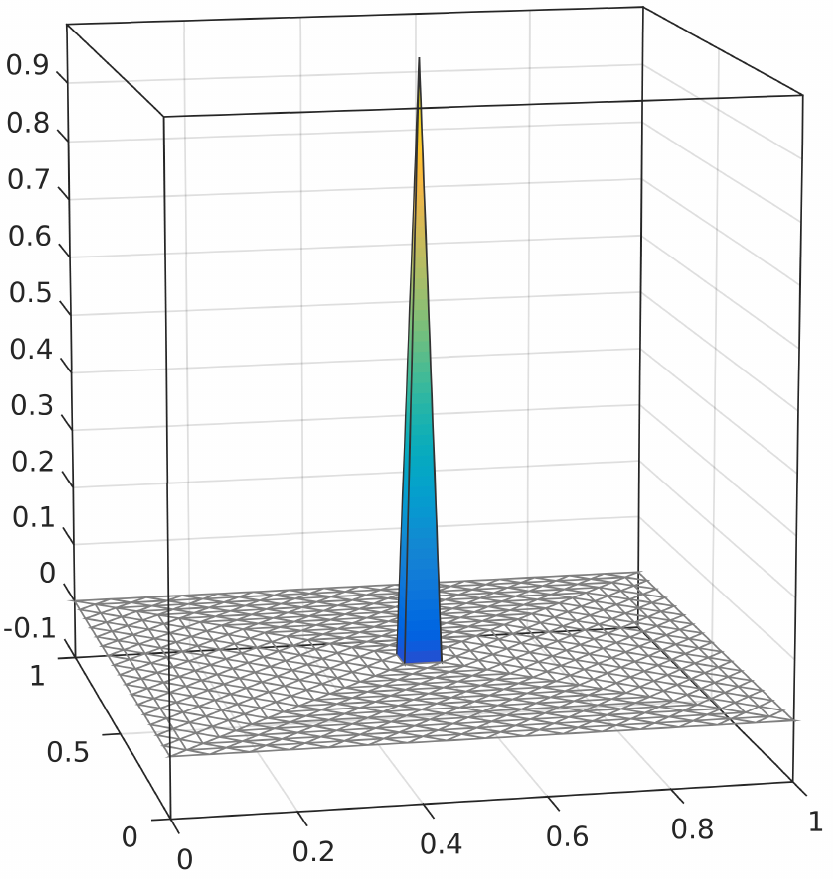}
\includegraphics[width=.5\textwidth]{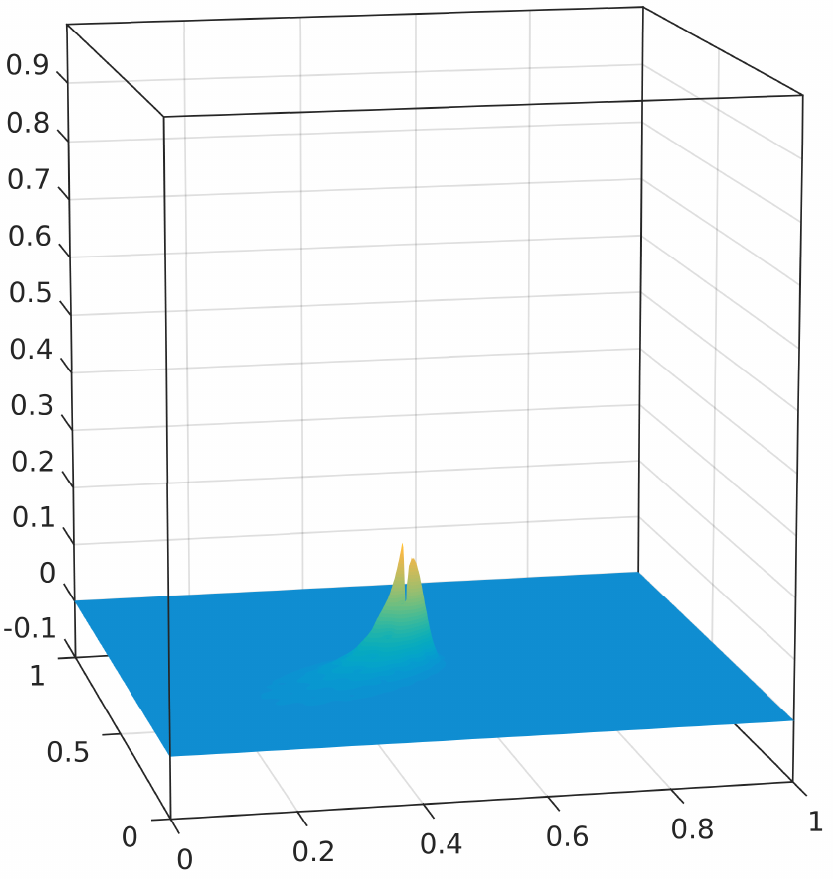}\\
\includegraphics[width=.5\textwidth]{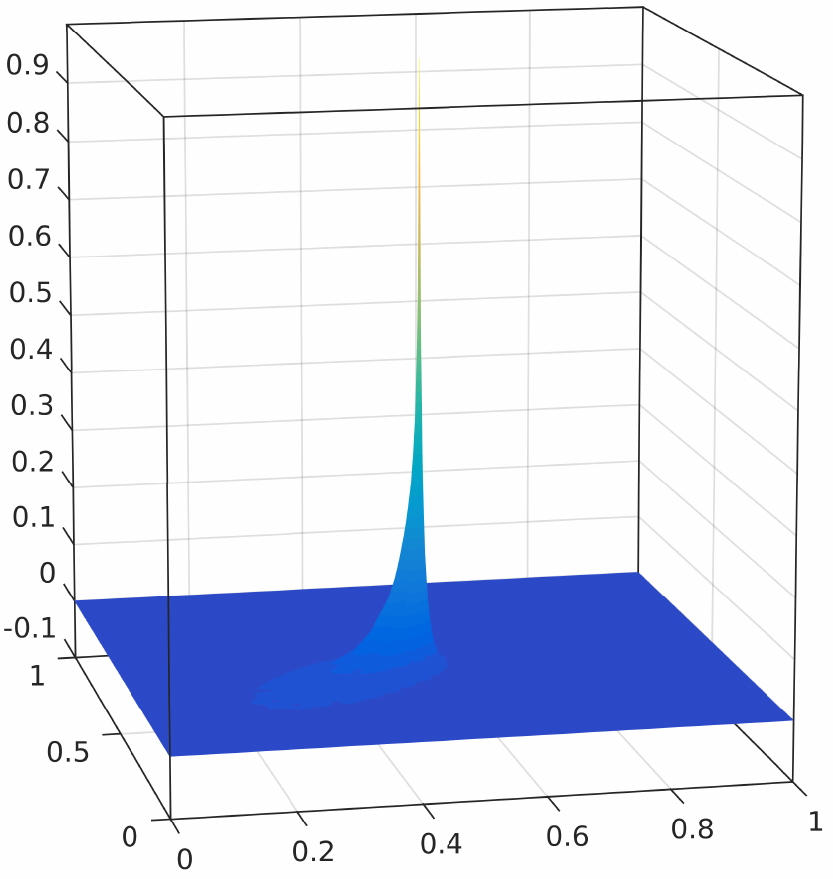}
\includegraphics[width=.5\textwidth]{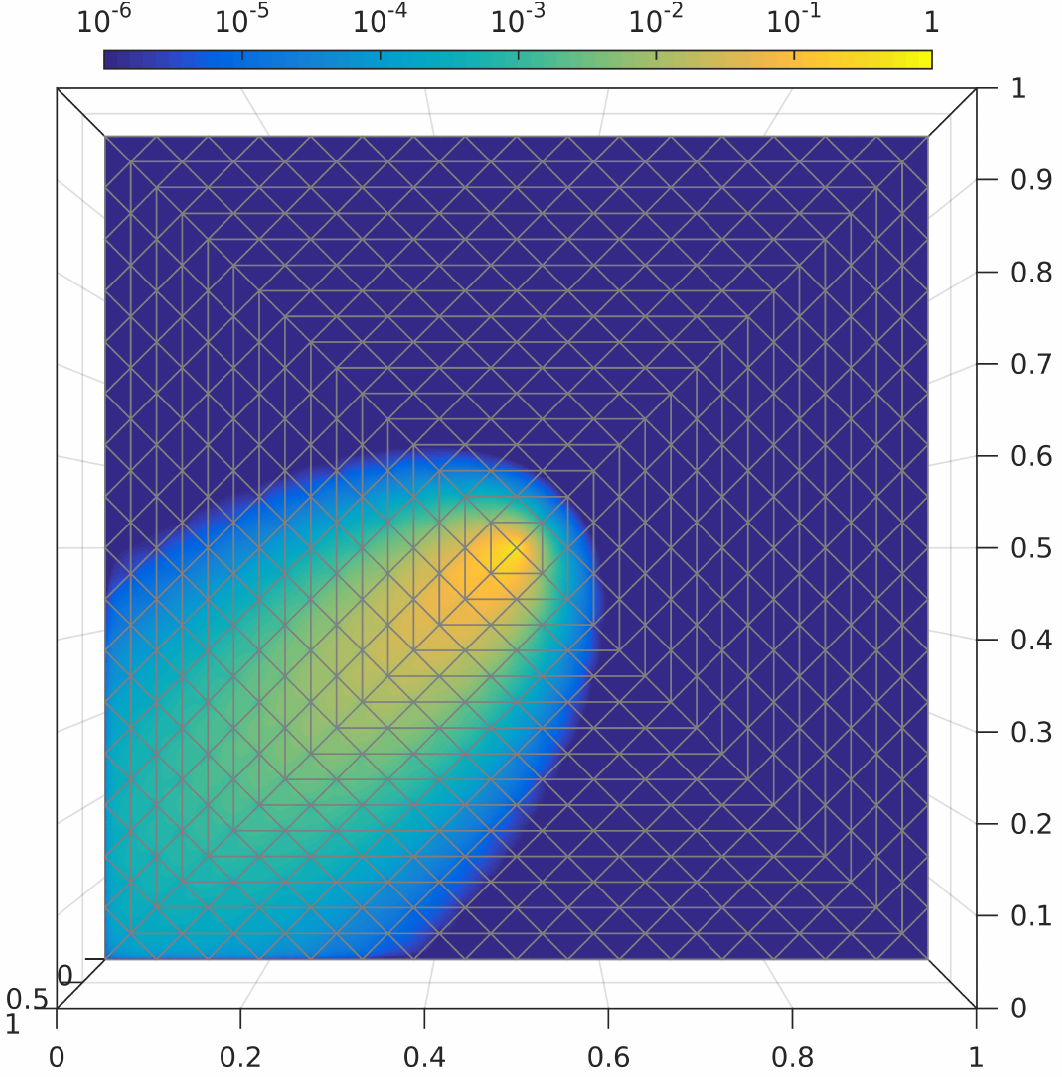}
%
\caption{Standard nodal basis function $\lambda_z$ with
respect to the coarse mesh $\tri_H$ (top left), corresponding ideal
corrector $\corrector  \lambda_z$ (top right), and corresponding
test basis function $(1-\corrector )\lambda_z$ (bottom left). 
The bottom right figure shows a top view on the modulus of test basis function $(1-\corrector)
\lambda_z$ with logarithmic color scale to illustrate the exponential decay
property. The underlying data is $b=[\cos(0.7),\sin(0.7)]$ and $\epsilon=2^{-7}$.}
\label{fig:testbasis_ideal}       
\end{figure}

Given $v_H\in V_H$, 
define the subscale corrector $\corrector:V_H\to\mathrm{Ker}I_H$ by
\begin{equation}\label{eqn:corrector2}
a(w,\corrector v_H)=a(w,v_H)
 \qquad\text{for all }w\in \mathrm{Ker}I_H.
\end{equation}
The well-posedness of~\eqref{eqn:corrector2} follows from the ellipticity 
and continuity of $a$, since $\mathrm{Ker}I_H\subset V$.

Now we are ready to define the multiscale test space as
\begin{align*}
W_{H}:=(1-\corrector)V_H.
\end{align*}
Note that~\eqref{eqn:corrector2} implies that
\[W_{H}=\{w\in V_h:\forall v\in \mathrm{Ker}I_H, a(v,w)=0\}.\]
The Petrov-Galerkin method for the approximation of \eqref{eqn:fine-scale} based on 
the trial-test pairing $(V_{H},W_{H})$ defined above seeks $u_{H}\in V_{H}$ satisfying
\begin{align}\label{eqn:ideal}
  a(u_{H},w_{H})=\langle f,w_{H}\rangle_{H^{-1}(\Omega)\times H^1_0(\Omega)}
    \qquad \text{for all }w_{H}\in W_{H}.
\end{align}
%
Note that~\eqref{eqn:ideal} is a variational characterization of $I_H$
in the sense that, for all $w_{H}\in W_{H}$, we have 
\[a(I_H u_h,w_{H})=a(I_H u_h-u_h,w_{H})+a(u_h,w_{H})=
(f,w_{H}),\]
where the last equality follows from~\eqref{eqn:corrector2}, 
\eqref{eqn:fine-scale} and the fact that $I_H u_h-u_h\in \mathrm{Ker}I_H$.
Since $\mathrm{dim}V_H=\mathrm{dim}W_H$, it follows that $u_H=I_H u_h\in V_H$ 
is the unique solution of~\eqref{eqn:ideal}
and the ideal method inherits favourable stability 
and approximation properties from the interpolation $I_H$.
To be more precise, we have the following proposition,
which follows directly from the identity $u_H=I_H u_h$ and~\eqref{eqn:interpolation}.

\begin{proposition}[Stability and local quasi-optimality of the ideal method]\label{prop:ideal}
For any $f\in H^{-1}(\Omega)$, the ideal Petrov-Galerkin method \eqref{eqn:ideal} 
admits a unique solution $u_H$ in the standard finite element space $V_H$. 
The method is stable in the sense that
\[\normHsemi{u_H}{\Omega}\leq \CI\normHsemi{u_h}{\Omega},\]
where $u_h\in V_h$ denotes the reference solution that solves \eqref{eqn:fine-scale}. 
Note that the constant $\CI$ is independent of $\epsilon$, but may 
depend on $H/h$.

Moreover, for any $T\in \tri_H$, we have the local best-approximation 
result
\[\normHsemi{u_h-u_H}{T}\leq \CI\min\limits_{v_H\in V_H}\normHsemi{u_h-v_H}{T}.\]
\end{proposition}

\begin{remark}
The stability and quasi-optimality of Proposition~\ref{prop:ideal} also 
holds for any other norm in which $I_H$ is stable.
\end{remark}

We admit that the corrector problems~\eqref{eqn:corrector2} are global problems 
on the fine triangulation $\tri_h$ which have to be precomputed for the 
solving of~\eqref{eqn:ideal}. This would result in a number of 
$\mathrm{dim}(V_H)$ problems of dimension $O(\mathrm{dim}(V_h))$ which is 
comparable of solving the original problem \eqref{eq:original} 
on a fine grid by an efficient standard method.
This makes the VMS~\eqref{eqn:ideal} not realistic.
However, it can be observed in Figure~\ref{fig:testbasis_ideal} that the 
corrector of functions with local support are still quasi-local in the sense 
that they decay exponentially. This allows for an approximation of 
the corrector by functions of local support.
In the next section, the exponential decay will be made rigorous, while 
Section~\ref{s:lod} proves stability and approximation properties 
for a localization strategy.

We end this section with a proof of the stability in the classical inf-sup sense,
although the method is perfectly stable in the sense of Proposition~\ref{prop:ideal}.
This result will be used in Section~\ref{s:lod} to prove well posedness of the 
localized version of~\eqref{eqn:ideal}.

\begin{lemma}[Stability]\label{lemma:infsup}
The trial-test pairing $(V_{H},W_{H})$ satisfies the inf-sup condition 
\begin{align}\label{eqn:stability}
\inf\limits_{w_{H}\in W_{H}\setminus\{0\}}
     \sup\limits_{u_{H}\in V_{H}\setminus\{0\}}
            \frac{a(u_{H},w_{H} )}{\normHsemi{u_{H}}{\Omega}\normHsemi{w_{H}}{\Omega}}
     \geq  \frac{\epsilon}{\CI}.
\end{align}
\end{lemma}

\begin{proof}
Given $w_{H}\in W_{H}$, take $u_H=I_H(w_{H})\in V_H$. Then by
\eqref{eqn:interpolation},
\begin{align}\label{eqn:14}
\normHsemi{u_H}{\Omega}\leq \CI\normHsemi{w_{H}}{\Omega}.
\end{align}
Note that $I_H(w_{H})-w_{H}\in \mathrm{Ker}I_H$. By~\eqref{eqn:corrector2}, we have
\begin{align}
  a(u_{H},w_{H})=a(I_H(w_{H}),w_{H})=a(w_{H},w_{H})
    =\epsilon \normHsemi{w_{H}}{\Omega}^2,
\end{align}
where the last inequality follows from $\nabla\cdot b=0$.

We obtain the result by the application of \eqref{eqn:14}.
\end{proof}

\section{Exponential decay of element correctors}\label{s:exponentialdecay}

This section is devoted to the proof of the exponential decay of element 
correctors defined in the following. 
Given $\omega\subset\Omega$, define the local bilinear form
\begin{align}\label{eqn:defaloc}
 a_\omega(u,v):=\epsilon \int_\omega \nabla u\cdot\nabla v\,dx 
    + \int_\omega (b\cdot \nabla u)\, v\,dx 
    \qquad\text{for all }u,v\in V
\end{align}
and let the local corrector $\corrector_{T}:V_H\to \mathrm{Ker}I_H$ 
be defined for any $v_H\in V_H$ by 
\begin{align}\label{eqn:defloccor}
  a(w,\corrector_{T}v_H) = a_T(w,v_H)
  \qquad\text{for all }w\in\mathrm{Ker}I_H.
\end{align}
Note that $\corrector =\sum_{T\in\tri_H} \corrector_{T}$
holds for the corrector $\corrector $ defined by~\eqref{eqn:corrector2}.

We consider the case that $\epsilon\leq H$.
In the following we restrict ourselves to a constant vector field $b$ 
and w.l.o.g.\ $|b|=1$; see Remark~\ref{remark:varb} below for a discussion 
for non-constant vector fields $b$. 
Define $t$ as a unit vector in $\mathbb{R}^{2}$, s.t. $t\cdot b=0$. Define a rectangle 
$S_{T,\ell,b}$ for each $T\in \tri_H$ and $\ell\in \mathbb{N}_{+}$ by 
\begin{equation}\label{eqn:defStlb}
\begin{aligned}
 S_{T,\ell,b}&:= \Omega\cap \operatorname{conv}\{\mathrm{mid}_T-\ell Ht+\ell H b,
      \mathrm{mid}_T+\ell Ht+\ell H b, \\
  &\qquad\qquad\qquad\qquad
      \mathrm{mid}_T-\ell Ht-\ell H^2 b/\epsilon,
      \mathrm{mid}_T+\ell Ht-\ell H^2b/\epsilon\}.
\end{aligned}
\end{equation}
We do not assume that $b$ is aligned with the triangulation and therefore 
we define the patches $\Omega_{T,\ell,b}$ by
\begin{equation*}
\begin{aligned}
  \Omega_{T,\ell,b} := \cup \{T'\in\tri_H\mid T'\cap S_{T,\ell,b}\neq \emptyset\}
    \supset S_{T,\ell,b}.
\end{aligned}
\end{equation*}
See Figure~\ref{fig:patches} for an illustration.
For fixed $\ell \in \mathbb{N}_{+}$, the element patches have finite overlap in the 
sense that there exists a constant $\Col>0$, s.t., 
\begin{align}\label{eqn:overlap}
\max_{K\in\tri_H} \#\{T\in\tri_H\mid K\subset \Omega_{T,\ell,b}\}\leq \Col.
\end{align}

%

\begin{figure}[htb!]
\includegraphics[width=0.325\textwidth]{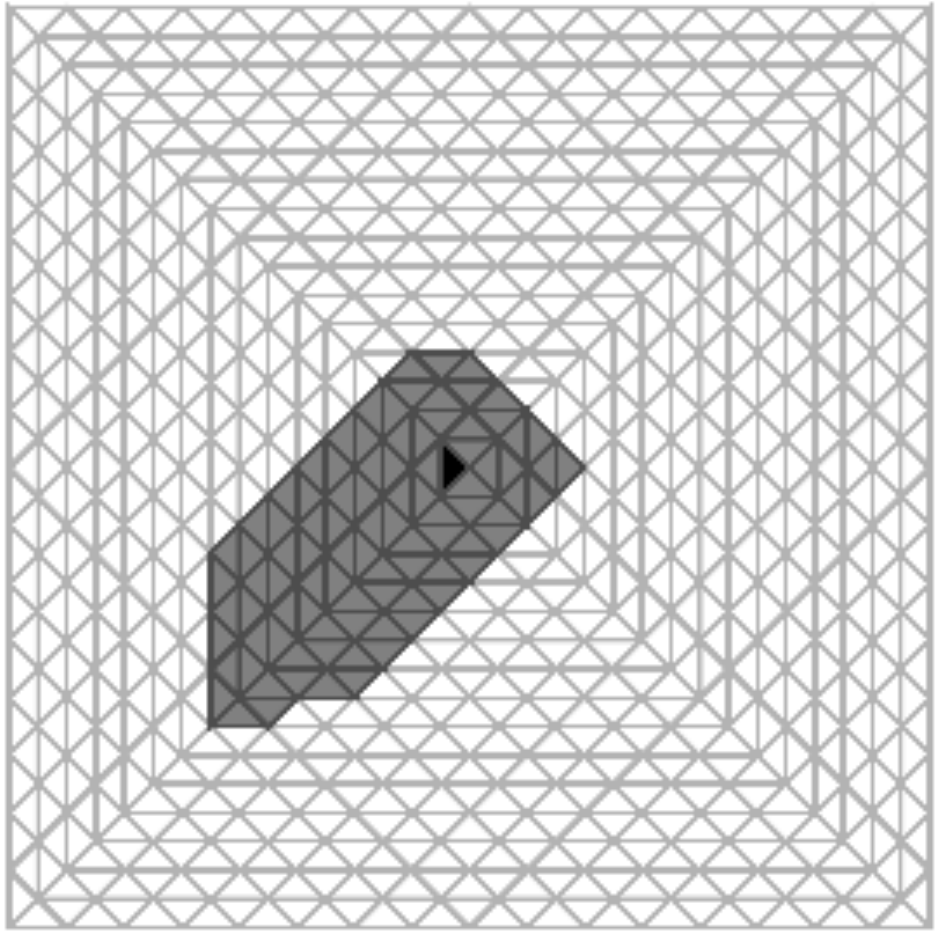}
\includegraphics[width=0.325\textwidth]{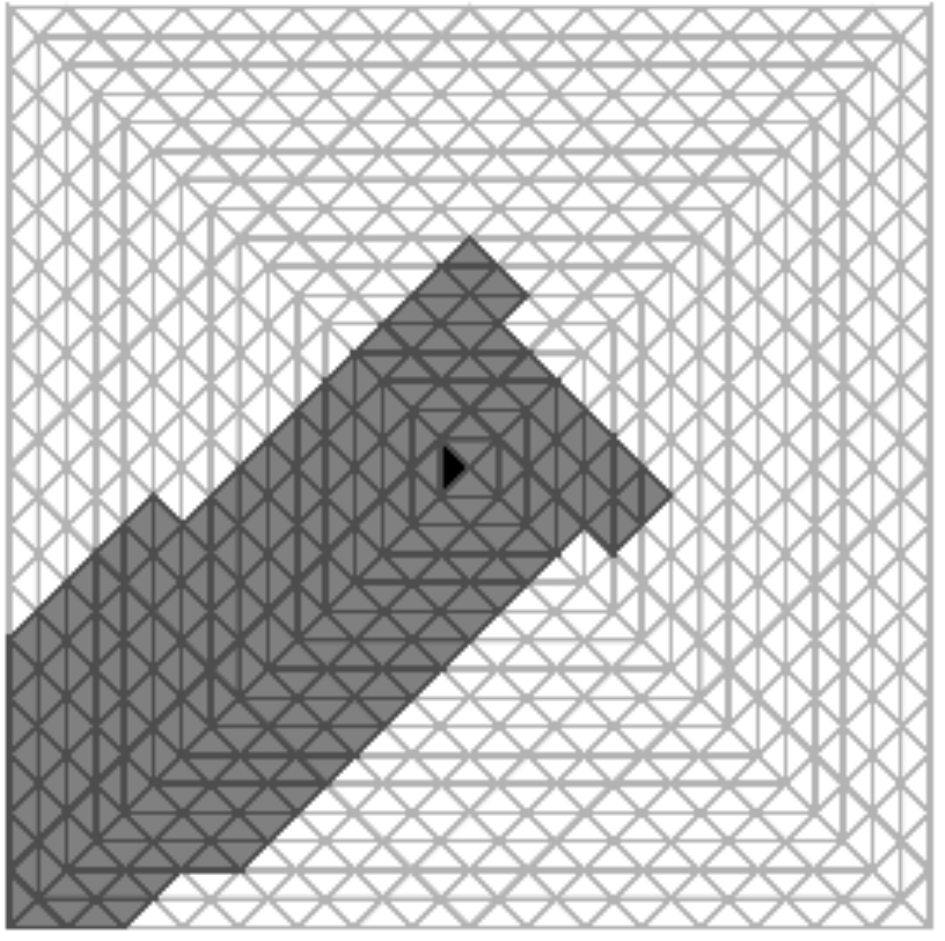}
\includegraphics[width=0.325\textwidth]{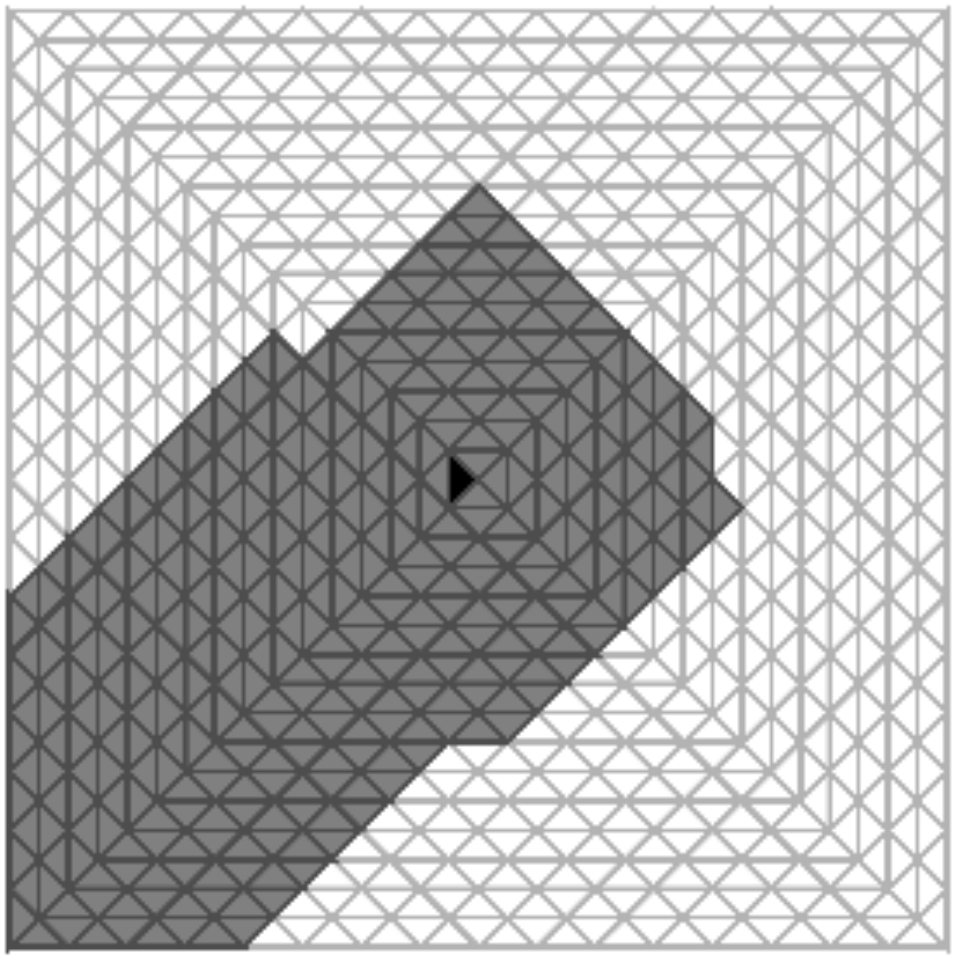}
\caption{Element patches $\Omega_{T,\ell,b}$ for $b=[\cos(0.7),\sin(0.7)]$, 
$\epsilon=2^{-7}$ and $\ell=1,2,3$ (from left to right) as they are used in the localised 
corrector problem \eqref{eqn:corrector1_local}.\label{fig:patches}}
\end{figure}

\begin{theorem}\label{thm:decay}
Let $T\in \tri_H$
and $v_H\in V_H$ and let $\corrector_{T} v_H$ denote the corresponding 
local subscale 
corrector as defined in~\eqref{eqn:defloccor}. 
Then we have
\begin{align}
\normHsemi{\corrector_{T} v_H}{\Omega\setminus S_{T,\ell,b}}
        \lesssim \beta^{\ell}\normHsemi{\corrector_{T} v_H}{\Omega}.
\end{align}
The constant $\beta$ reads 
\begin{align}\label{eqn:defbeta}
 \beta = \left(\frac{4\CI+3\CI^2}{1+4\CI+3\CI^2}\right)^{1/2}<1
\end{align}
and is bounded away from 1.
\end{theorem}

Before going to the proof of this theorem, we express the exponential 
decay in terms of patches in the following corollary. This is a direct consequence 
of Theorem~\ref{thm:decay} and the definition of $\Omega_{T,\ell,b}$.

\begin{corollary}
Let $T\in \tri_H$
and $v_H\in V_H$ and let $\corrector_{T} v_H$ denote the corresponding 
local subscale 
corrector as defined in~\eqref{eqn:defloccor}. 
Then we have
\begin{align}
\normHsemi{\corrector_{T} v_H}{\Omega\setminus \Omega_{T,\ell,b}}
\leq \normHsemi{\corrector_{T} v_H}{\Omega\setminus S_{T,\ell,b}}
        \lesssim \beta^{\ell}\normHsemi{\corrector_{T} v_H}{\Omega}
\end{align}
with $\beta<1$ from \eqref{eqn:defbeta}.
\end{corollary}

\begin{remark}
Recall that in this two-dimensional situation $\CI\lesssim \log(\tfrac{H}{h})$
for $\CI$ from~\eqref{eqn:interpolation}.
Then given fixed $H\in (0,1)$ and let $h\to 0$,  the constant $\beta$ scales as 
\begin{align*}
  1-\beta^2 \gtrsim \frac{1}{1+\log(h)^2}.
\end{align*}
In the three dimensional case, Theorem~\ref{thm:decay} could essentially 
be proven in the same way, but the dependence of $\CI$ on $H/h$ is algebraic 
so that the decay rate deteriorates very fast.
\end{remark}

\begin{proof}[Proof of Theorem~\ref{thm:decay}]
The crucial point in the proof is~\eqref{eqn:estbeta} below, which exploits 
the direction of $b$. This allows for patches that are only enlarged 
in the direction of $-b$.
The remaining part of the proof then essentially follows as in~\cite{Mlqvist.Peterseim:2011}.

Define 
a cut-off function
\[\eta:=1 - \eta_1\eta_2,\]
where $0\leq\eta_1(x)\leq 1$ and $0\leq\eta_2(x)\leq 1$ are one-dimensional 
continuous piecewise affine cut-off functions along $t$ and $b$, respectively. 
Recall that $\mathrm{mid}_T$ denotes the baricenter of a coarse element $T$, 
$|b|=1$ and $t$ is a unit vector orthogonal to $b$. 
We define $\eta_1$ and $\eta_2$ by
\begin{equation}\label{eqn:defeta1}
\eta_1(x)=
\left\{ \begin{aligned}
         1 &\quad\text{ if }|(x-\mathrm{mid}_T)\cdot t|\leq (\ell-1) H; \\
         0 &\quad\text{ if }|(x-\mathrm{mid}_T)\cdot t|\geq \ell H
\end{aligned} \right.
\end{equation}
and
\begin{equation}\label{eqn:defeta2}
\eta_2(x)=
\left\{ \begin{aligned}
         1 &\quad\text{ if }-(\ell-1) H\leq -(x-\mathrm{mid}_T)\cdot b\leq (\ell-1)\frac{H^2}{\epsilon}; \\
         0 &\quad\text{ if }-(x-\mathrm{mid}_T)\cdot b\geq \ell\frac{H^2}{\epsilon} 
                           \;\text{ or }\;-(x-\mathrm{mid}_T)\cdot b\leq -\ell H.
\end{aligned} \right.
\end{equation}
We obtain from the construction above that $\nabla\eta_1(x)\cdot b=0$ for all 
$x\in\Omega$ and $\eta_1\leq 1$.
Moreover, since $-(b\cdot \nabla \eta_2(x))\leq 0$ if 
$0\leq (x-\mathrm{mid}_T)\cdot b$, we deduce
\begin{align}\label{eqn:estbeta}
  -b\cdot\nabla\eta=-(b\cdot\nabla\eta_1)\eta_2-(b\cdot\nabla\eta_2)\eta_1
  =-(b\cdot\nabla\eta_2)\eta_1\leq \frac{\epsilon}{H^2}.
\end{align}
Furthermore, $\eta\vert_{S_{T,\ell-1,b}}=0$ and 
$\eta\vert_{\Omega\setminus S_{T,\ell,b}}=1$, and $\eta$ is bounded between 
0 and 1 and satisfies the Lipschitz continuity
\begin{align}\label{eqn:Lipschitzeta}
  \normI{\nabla\eta}{\Omega}\leq 2/H.
\end{align}
Note that $\mathrm{supp}(\nabla\eta)\subset S_{T,\ell,b}\setminus S_{T,\ell-1,b}$.

Let $(\bullet,\bullet):=(\bullet,\bullet)_{L^2(\Omega)}$ denote the 
$L^2$ scalar product and define $\varphi:=\corrector_{T} v_H$. 
Due to $\nabla\cdot b=0$, we have
$(b\cdot\nabla(\eta\varphi),\eta\varphi)=0$, and
\begin{align*}
\epsilon\normHsemi{\varphi}{\Omega\setminus S_{T,\ell,b}}^2
&\leq\epsilon(\nabla(\eta\varphi),\nabla(\eta\varphi)) +(b\cdot\nabla(\eta\varphi),\eta\varphi)\\
&=\epsilon(\nabla\varphi,\eta\nabla(\eta\varphi))
       +\epsilon(\nabla\eta,\varphi\nabla(\eta\varphi))
       +(b\cdot\nabla(\eta\varphi),\eta\varphi)\\
&= \epsilon(\nabla\varphi,\nabla(\eta^2\varphi))+(b\cdot\nabla(\eta^2\varphi),\varphi)
            -\epsilon(\nabla\varphi,\eta\varphi\nabla\eta)\\
&\qquad\qquad +\epsilon(\nabla\eta,\varphi\nabla(\eta\varphi))
            -(b\cdot\nabla\eta,\eta\varphi^2).
\end{align*}
Observe that $\eta^2\varphi\in \mathrm{Ker}I_H$, and we obtain
\begin{align}\label{eqn:12}
\epsilon(\nabla\varphi,\nabla(\eta^2\varphi))
  +(b\cdot\nabla(\eta^2\varphi),\varphi)=
a(\eta^2\varphi,\varphi)=
  a_T(\eta^2\varphi,v_H)=0
\end{align}
by the definition of $\corrector_{T} $ in \eqref{eqn:defloccor}.
Thus, we arrive at
 \begin{align}\label{eqn:termsEs}
&\epsilon\normHsemi{\varphi}{\Omega\setminus S_{T,\ell,b}}^2
\leq
 \epsilon\lvert(\nabla\varphi,\eta\varphi\nabla\eta)\rvert 
+\epsilon\lvert(\nabla\eta,\varphi\nabla(\eta\varphi))\rvert
-(b\cdot\nabla\eta,\eta\varphi^2).
\end{align}
We will estimate each term on the right hand side of \eqref{eqn:termsEs}.
With $\eta\leq 1$ and~\eqref{eqn:Lipschitzeta}, a Cauchy inequality leads to
\begin{align*}
  \epsilon|(\nabla\varphi,\eta\varphi\nabla\eta)|
  &\leq 2\epsilon H^{-1}\normHsemi{\varphi}{S_{T,\ell,b}\setminus S_{T,\ell-1,b}}
  \normL{\varphi}{S_{T,\ell,b}\setminus S_{T,\ell-1,b}}\\
  &\leq 2 \CI\epsilon\normHsemi{\varphi}{S_{T,\ell,b}\setminus S_{T,\ell-1,b}}^2,
\end{align*}
where we have used the fact that $\varphi\in \mathrm{Ker}I_H$ and 
estimate \eqref{eqn:interpolation} in the last inequality.

The same arguments imply for the second term in~\eqref{eqn:termsEs}
\begin{align*}
 \epsilon|(\nabla\eta,\varphi\nabla(\eta\varphi))|
  &\leq 2\epsilon H^{-1}
      \normL{\varphi}{S_{T,\ell,b}\setminus S_{T,\ell-1,b}}
     \normHsemi{\eta\varphi}{S_{T,\ell,b}\setminus S_{T,\ell-1,b}}\\
  &\leq 2 \epsilon \left(\CI^2+\CI\right)
       \normHsemi{\varphi}{S_{T,\ell,b}\setminus S_{T,\ell-1,b}}^2.
\end{align*}
The crucial point in the estimation of the last term in~\eqref{eqn:termsEs}
is the estimate~\eqref{eqn:estbeta}, which implies together with 
$\eta\varphi^2\geq 0$
\begin{align*}
 -(b\cdot\nabla\eta,\eta\varphi^2)
 &\leq \frac{\epsilon}{H^2}
   \normL{\varphi}{S_{T,\ell,b}\setminus S_{T,\ell-1,b}}^2\\
 &\leq \epsilon \CI^2\normHsemi{\varphi}{S_{T,\ell,b}\setminus S_{T,\ell-1,b}}^2.
\end{align*}
%
%
%
Assemble all estimates above for \eqref{eqn:termsEs}, to conclude
\begin{align*}
  &\epsilon\normHsemi{\varphi}{\Omega\setminus S_{T,\ell,b}}^2
    \leq \epsilon \left(4\CI+3\CI^2\right)
      \normHsemi{\varphi}{S_{T,\ell,b}\setminus S_{T,\ell-1,b}}^2.
\end{align*}
Define $C\hspace{-0.6ex}\left(\tfrac{H}{h}\right):=4\CI+3\CI^2$, which leads to 
\begin{align*}
  \normHsemi{\varphi}{\Omega\setminus S_{T,\ell,b}}^2
    \leq C\hspace{-0.6ex}\left(\tfrac{H}{h}\right)
      (\normHsemi{\varphi}{\Omega\setminus S_{T,\ell-1,b}}^2
                   -\normHsemi{\varphi}{\Omega\setminus S_{T,\ell,b}}^2),
\end{align*}
and therefore
\begin{align*}
  &\normHsemi{\varphi}{\Omega\setminus S_{T,\ell,b}}^2
    \leq \frac{C\hspace{-0.6ex}\left(\tfrac{H}{h}\right)}{1+C\hspace{-0.6ex}\left(\tfrac{H}{h}\right)}
       \normHsemi{\varphi}{\Omega\setminus S_{T,\ell-1,b}}^2.
\end{align*}
Repeating this process, we derive at
\begin{align*}
  &\normHsemi{\varphi}{\Omega\setminus S_{T,\ell,b}}^2
    \leq \left(\frac{C\hspace{-0.6ex}\left(\tfrac{H}{h}\right)}
              {1+C\hspace{-0.6ex}\left(\tfrac{H}{h}\right)}\right)^{\ell}
              \normHsemi{\varphi}{\Omega}^2.
\end{align*}
This concludes the proof.
\end{proof}

\begin{remark}[non-constant $b$]\label{remark:varb}
If the velocity field $b$ is divergence-free, but not globally constant, the definition of the 
rectangles $S_{T,\ell,b}$ has to be modified in that they have to follow 
the velocity. 

This should be made more precise in the situation that there exists a 
bounded diffeomorphism with bounded inverse that maps a constant reference 
velocity field $b_\mathrm{ref}$ to $b$, in the following sense.
Assume that there exists a reference domain $\Omega_\mathrm{ref}$ and 
a diffeomorphism $\psi:\Omega_{\mathrm{ref}}\to \Omega$, 
$\psi\in \mathcal{C}^1(\Omega_{\mathrm{ref}})$, such that 
\begin{align*}
  D \psi(y) b_\mathrm{ref} = b(\psi(y))
  \qquad\text{for all }y\in \Omega_\mathrm{ref}.
\end{align*}
The domain $S_{T,\ell,b}$ (formerly a rectangle) is then defined as 
$S_{T,\ell,b}:=\psi(S_{\mathrm{ref},T,\ell,b_\mathrm{ref}})$, where 
$S_{\mathrm{ref},T,\ell,b_\mathrm{ref}}\subset \Omega_\mathrm{ref}$ is 
defined for the constant vector field $b_\mathrm{ref}$ as 
in~\eqref{eqn:defStlb}.
The cut-off function $\eta=1-\eta_1 \eta_2$ is then defined by 
$\eta_j(x):=\eta_{j,\mathrm{ref}}(\psi^{-1}(x))$ for $\eta_{j,\mathrm{ref}}$ 
defined as in~\eqref{eqn:defeta1}--\eqref{eqn:defeta2}.
The boundedness of $D\psi^{-1}$ then proves 
\begin{align*}
 \|\nabla \eta\|_{L^\infty(\Omega)} \lesssim  H^{-1}.
\end{align*}
The definitions of $\eta_1$ and $\eta_2$ lead for all $x\in\Omega$ to 
\begin{align*}
  b(x)\cdot\nabla \eta_j (x)
    = \nabla \eta_{j,\mathrm{ref}}\vert_{\psi^{-1}(x)}
        \cdot (D\psi^{-1}(x) b(x))
    = \nabla \eta_{j,\mathrm{ref}}\vert_{\psi^{-1}(x)}
        \cdot b_\mathrm{ref},
\end{align*}
which implies $-b\cdot\nabla\eta\leq \epsilon/H^2$.
Theorem~\ref{thm:decay} then follows as before.

\begin{figure}
  \begin{center}
    \includegraphics[width=.49\textwidth]{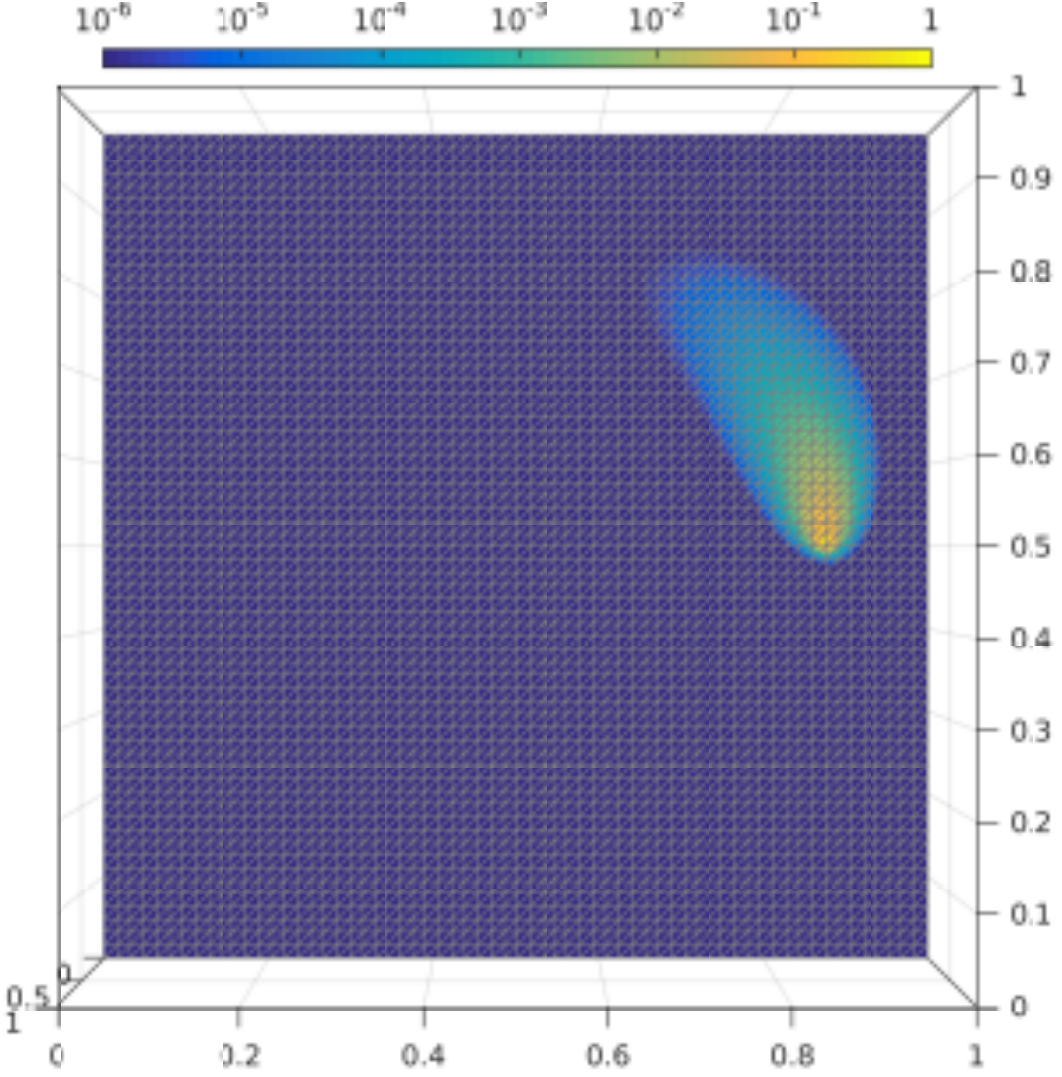}
    \includegraphics[width=.49\textwidth]{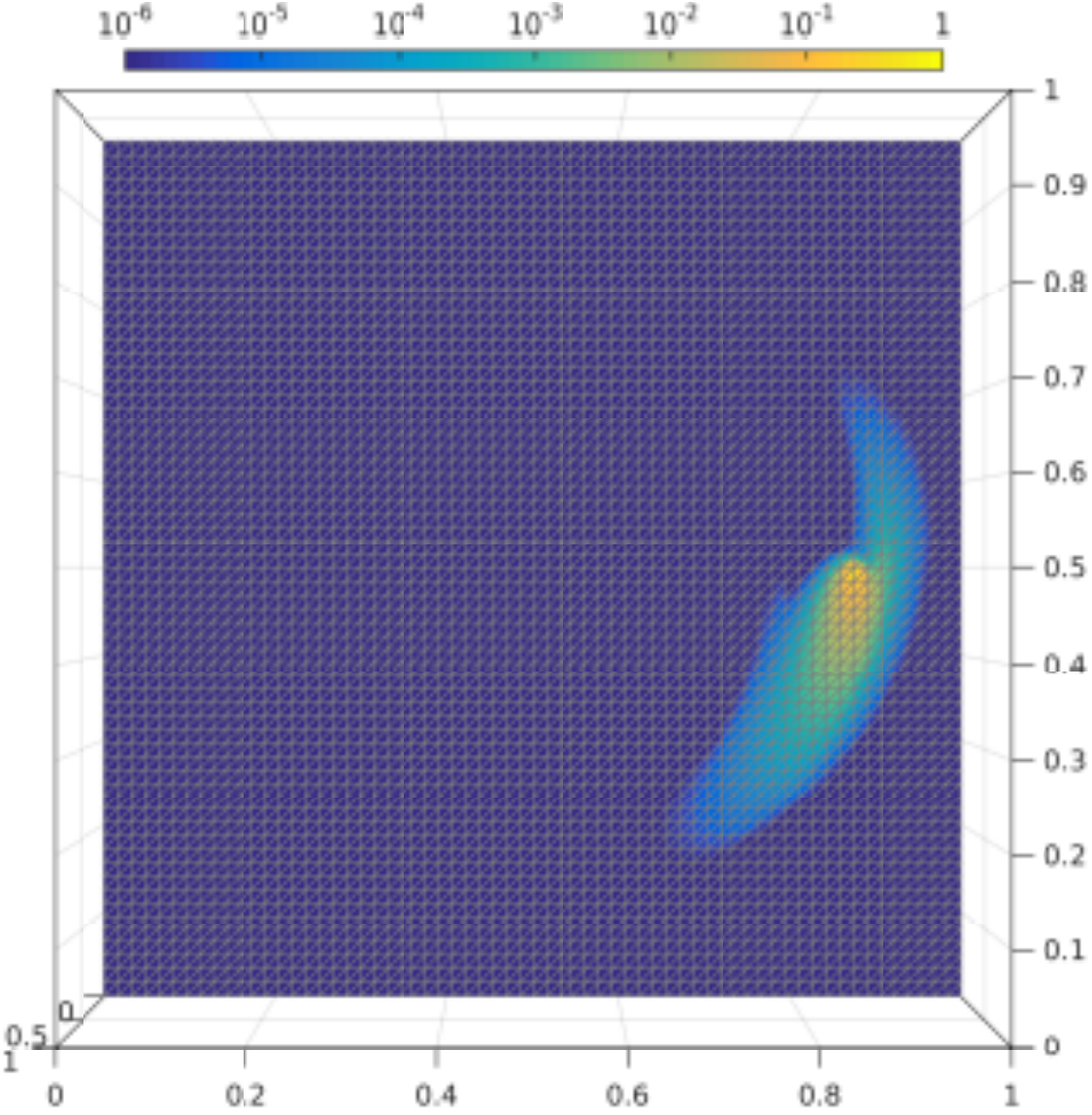}
  \end{center}
  \caption{\label{fig:correctorvarb}Top view on the modulus of (ideal) test 
    basis functions $(1-\corrector)\lambda_z$ with logarithmic color scale.
    The underlying data are $\epsilon=2^{-8}$ and $b_1$ (left) and 
    $b_2$ (right) from~\eqref{eqn:varb12}.}
\end{figure}

Figure~\ref{fig:correctorvarb} displays the modified basis 
functions $(1-\corrector)\lambda_z$ for the following non-constant vector fields 
\begin{align}\label{eqn:varb12}
  b_1(x) = 5\begin{pmatrix} x_2-0.5 \\ 0.5 - x_1 \end{pmatrix}
  \;\text{and}\;\;
  b_2(x) =  \left(2\left(\left\lceil x-\begin{pmatrix} 0.5\\0.5\end{pmatrix}\right\rceil-1\right) 
            \operatorname{mod} 2\right) b_1(x),
\end{align}
where $\lceil r\rceil:=\min\{k\in\mathbb{N}\mid k\geq r\}$ denotes the 
ceiling function.
One observes that the decay is directed along $b$.
\end{remark}

\section{LOD method and error analysis}\label{s:lod}

Based on the results above, we conclude that the energy norm of $\corrector_T v$ 
decreases very fast outside of a local region around $T$ for any $v\in V_H$. 
Therefore, a localization process is feasible to reduce the computational costs 
of the ideal method but maintain a good accuracy. 
In this section, we want to localize the corrector problems \eqref{eqn:corrector2}.  To this end, 
instead of solving them on the global domain $\Omega$, we obtain a 
good approximation of those correctors by solving a local problem on $\Omega_{T,\ell,b}$. 

Firstly, let us introduce some notations. In the following, we will denote 
$R_H=\mathrm{Ker}I_H$, and 
$R_H(\Omega_{T,\ell,b})=\{w\in R_H,\text{ and }w=0 \text{ in }\Omega\setminus \Omega_{T,\ell,b}\}$.
Recall the local bilinear form $a_\omega$ defined in~\eqref{eqn:defaloc}.
The localized element corrector $\corrector_{T,\ell} : V_H\rightarrow R_H(\Omega_{T,\ell,b})$
is defined as follows: given $v_H\in V_H$, let 
$\corrector_{T,\ell}v_H\in R_H(\Omega_{T,\ell,b})$ satisfy
\begin{equation}\label{eqn:corrector1_local}
a_{\Omega_{T,\ell,b}}(w,\corrector_{T,\ell}v_H)=a_T(w,v_H)
 \qquad \text{for all }w\in R_H(\Omega_{T,\ell,b}).
\end{equation}
Then we denote $\corrector_{\ell} :=\sum\limits_{T\in \tri_H}\corrector_{T,\ell}$;
see Figure \ref{fig:testbasis} for 
an illustration of the localized correctors $\corrector_{\ell} \lambda_z$ and the 
corresponding localized test basis.

\begin{figure}[t]
\includegraphics[width=.245\textwidth]{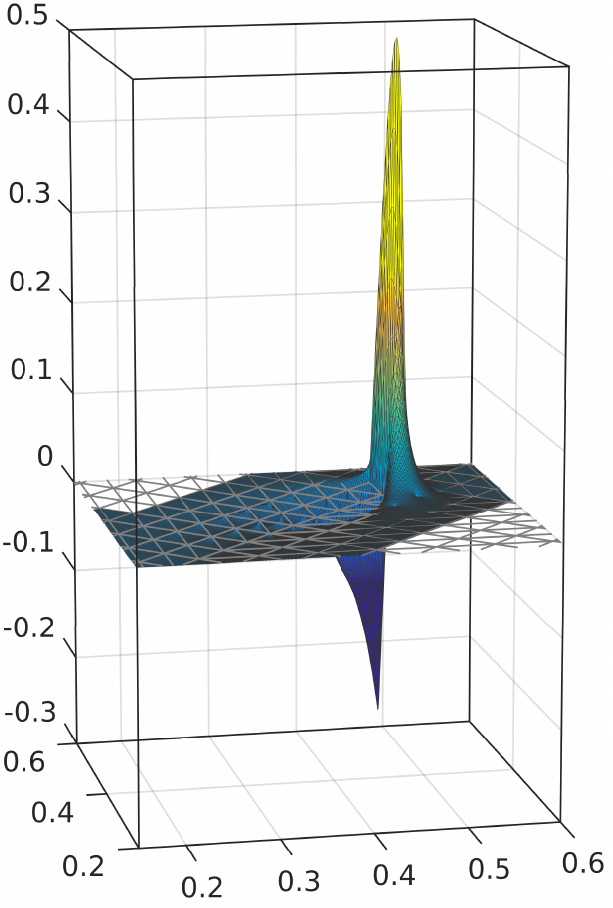}
\includegraphics[width=.245\textwidth]{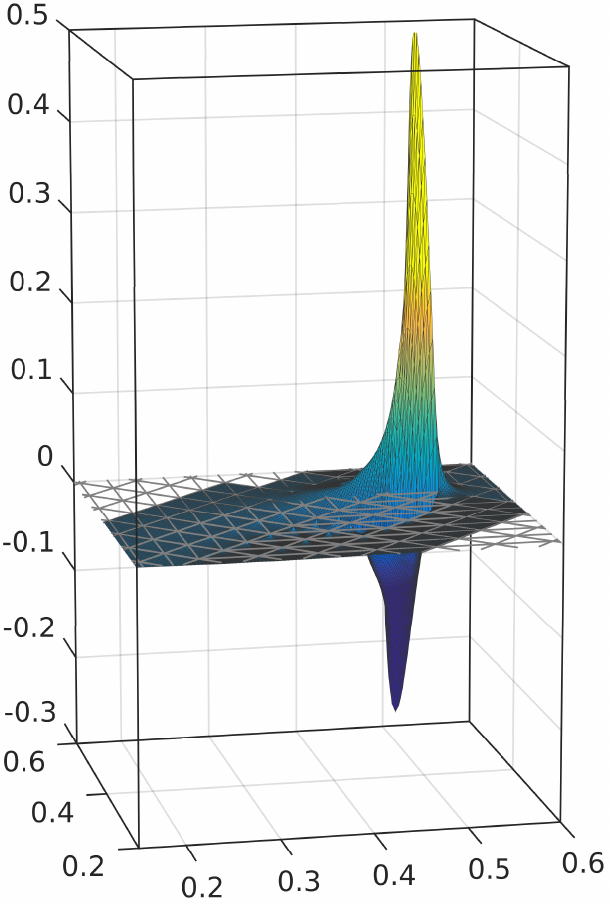}
\includegraphics[width=.245\textwidth]{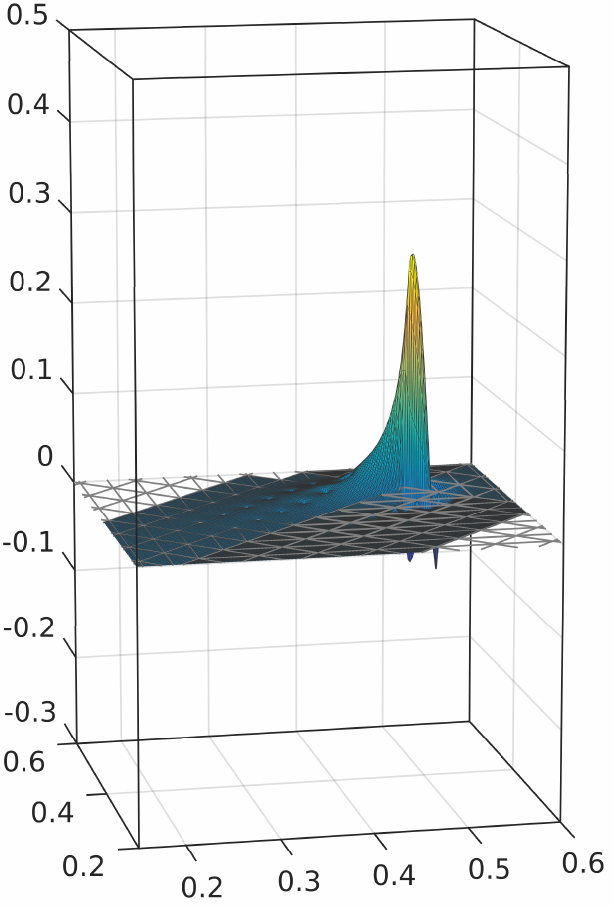}
\includegraphics[width=.245\textwidth]{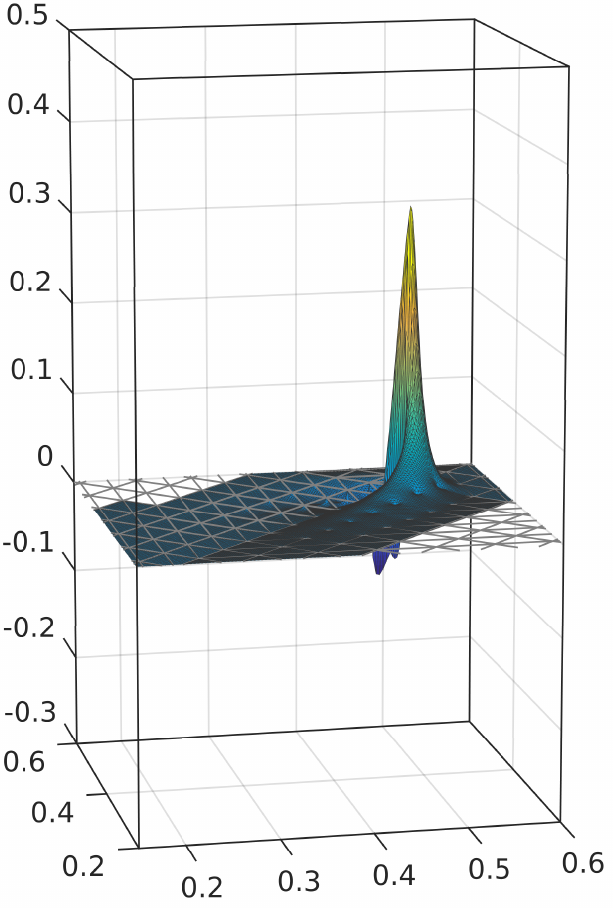}\vspace{1ex}\\
\includegraphics[width=.5\textwidth]{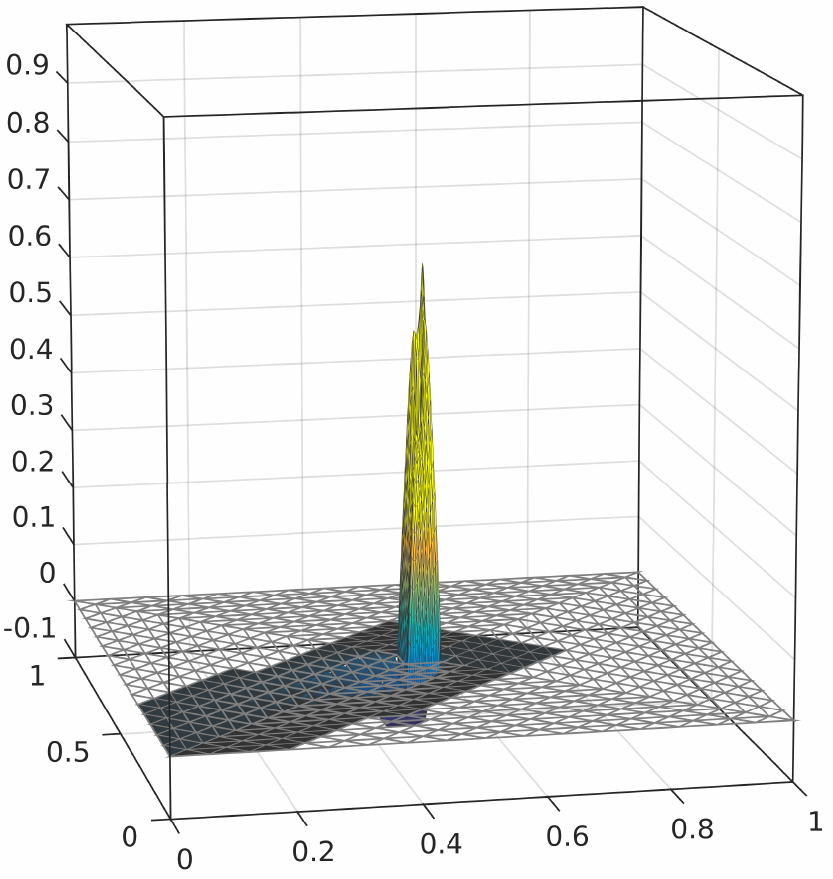}
\includegraphics[width=.5\textwidth]{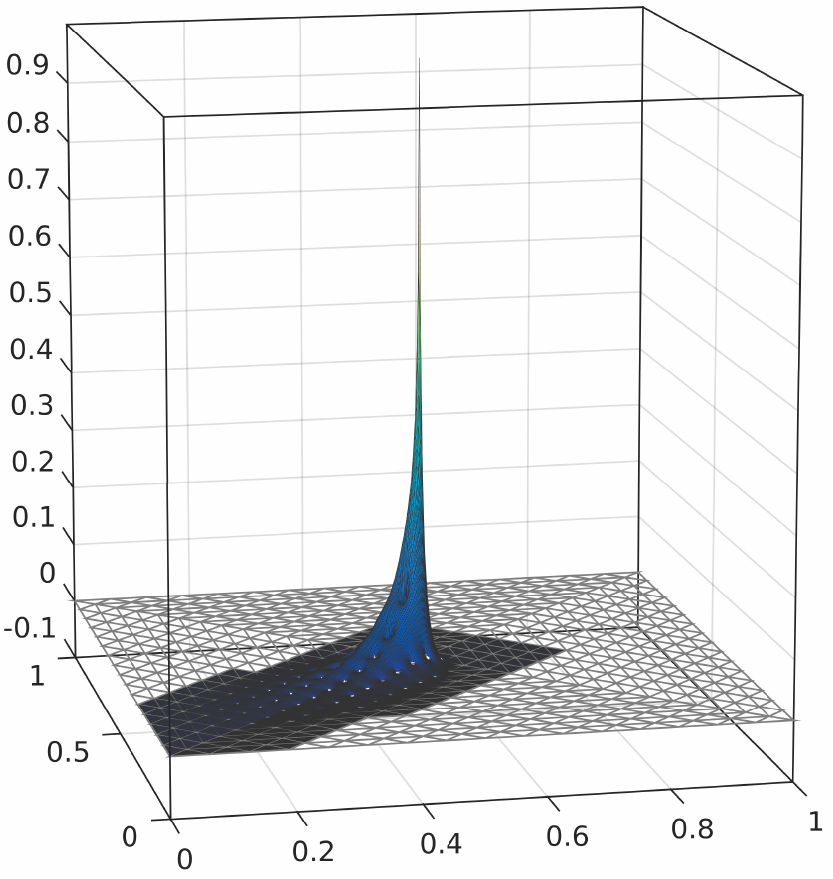}%
%
\caption{Localized element correctors $\corrector_{T,\ell}\lambda_z$ for $\ell=2$
and all four elements $T$ adjacent to the vertex $z=[0.5,0.5]$ (top),
localized nodal corrector $\corrector_{\ell}\lambda_z=\sum_{T\ni z}
\corrector_{T,\ell}\lambda_z$ (bottom left) and corresponding test basis function $
(1-\corrector_{\ell})\lambda_z$ (bottom right). The underlying data is
$b=[\cos(0.7),\sin(0.7)]$ and $\epsilon=2^{-7}$. The computations have been
performed by standard linear finite elements on local fine meshes of
width $h=2^{-8}$. See Fig.~\ref{fig:testbasis_ideal} for
a comparison with the ideal global  corrector and basis.}
\label{fig:testbasis}       
\end{figure}

In the following lemma, we will show that $\corrector_{T,\ell}$ is a good approximation 
of $\corrector_{T} $ provided that the local patches $\Omega_{T,\ell,b}$ are sufficiently large.
For the ease of presentation, we denote the mesh P\'eclet number $\Peclet{H}$ of $\tri_H$ by
\begin{align}\label{eqn:defPeclet}
  \Peclet{H}:=H\|b\|_{L^\infty(\Omega)}/\epsilon.
\end{align}
Recall the definition of $\beta$ from~\eqref{eqn:defbeta}.

\begin{lemma}\label{lemma:LocalGlobal}
Given $v\in V_H$ and $\ell\in\mathbb{N}_{+}$, it holds that
\begin{align}\label{eqn:LocalGlobal}
  \normHsemi{\corrector_{T} v-\corrector_{T,\ell} v}{\Omega}
   \lesssim \left( 1+\Peclet{H}\CI\right)^2 
        (\CI+1)  \beta^{\ell-1}
   \normHsemi{v}{T}.
\end{align}
\end{lemma}

\begin{proof}
Denote $e_{T,\ell}:=\corrector_{T} v-\corrector_{T,\ell} v$. In view of
$R_{H}(\Omega_{T,\ell,b})\subset R_{H}$, the definitions of the correctors 
in~\eqref{eqn:corrector1_local} and \eqref{eqn:corrector2} and 
the orthogonality of Petrov-Galerkin type, lead to
\begin{align*}
  \epsilon\normHsemi{e_{T,\ell}}{\Omega}^2
  =a(e_{T,\ell},e_{T,\ell})=a(e_{T,\ell}-w,e_{T,\ell})
  \quad\text{ for all }w\in R_{H}(\Omega_{T,\ell,b}).
\end{align*}
Since $I_H(e_{T,\ell})=0$, H\"{o}lder's inequality, 
and the approximation property~\eqref{eqn:interpolation} of $I_H$ imply  
\begin{align*}
 \normHsemi{e_{T,\ell}}{\Omega}^2
  \leq \left( 1+\Peclet{H}\CI\right) 
         \normHsemi{e_{T,\ell}}{\Omega}
    \normHsemi{e_{T,\ell}-w}{\Omega}.
\end{align*}
Since $w\in R_{H}(\Omega_{T,\ell,b})$ is arbitrary, we arrive at
\begin{align}\label{eqn:333}
  \normHsemi{e_{T,\ell}}{\Omega}
      \leq \left( 1+\Peclet{H}\CI\right)  
         \normHsemi{\corrector_{T}v-w}{\Omega}.
\end{align}
In the following, we construct a specific $w\in R_H(\Omega_{T,\ell,b})$ to control the term 
$\normHsemi{\corrector_{T}v-w}{\Omega}$. 
Let $\eta$ denote the cut-off function from the proof of Theorem~\ref{thm:decay},
such that $\eta\vert_{S_{T,\ell-1,b}}=0$ and 
$\eta\vert_{\Omega\setminus S_{T,\ell,b}}=1$. Note that 
$S_{T,\ell,b}\subset\Omega_{T,\ell,b}$ and therefore 
$\mu:=1-\eta$ satisfies $\mu\vert_{\Omega\setminus\Omega_{T,\ell,b}}=0$.
In addition, $\mu$ is bounded between 0 and 1 and satisfies the Lip\-schitz continuity
\begin{align}\label{eqn:scalingmu}
  \normI{\nabla\mu}{\Omega}\leq 2 H^{-1}.
\end{align}
Define $w=\mu\corrector_{T}v$, then $w\in R_H(\Omega_{T,\ell,b})$.
Since $\corrector_{T}v\in R_H$, the fact that $0\leq\mu\leq 1$ 
and~\eqref{eqn:scalingmu} lead as in the proof of Theorem~\ref{thm:decay} to 
\begin{align*}
  \normHsemi{\corrector_{T}v-w}{\Omega}
    &=\normHsemi{\corrector_{T}v-\mu\corrector_{T}v}{\Omega\setminus S_{T,\ell-1,b}}\\
  &\leq 2 (\CI+1)\normHsemi{\corrector_{T}v}{\Omega\setminus S_{T,\ell-1,b}}.
\end{align*}
Theorem~\ref{thm:decay} then implies
\begin{align*}
  \normHsemi{\corrector_{T}v-w}{\Omega}
    &\lesssim (\CI+1)\beta^{\ell-1}\normHsemi{\corrector_{T} v}{\Omega}.
\end{align*}
The combination with~\eqref{eqn:333} implies 
\begin{align*}
  \normHsemi{\corrector_{T}v-\corrector_{T,\ell}v}{\Omega}
  \lesssim \left( 1+\Peclet{H}\CI\right) (\CI+1)\beta^{\ell-1}
      \normHsemi{\corrector_{T}v}{\Omega}.
\end{align*}
In the end, we show the stability of $\mathcal{C}_{T}$ to bound the term 
$\normHsemi{\corrector_{T}v}{\Omega}$. Since $I_H(\corrector_{T}v)=0$, the 
stability of $\corrector_{T}$ follows from 
\begin{align*}
 \epsilon \normHsemi{\corrector_{T} v}{\Omega}^2
   &= a(\corrector_{T} v,\corrector_{T} v)
   = a_T(\corrector_{T} v, v)\\
   &\leq \epsilon \normHsemi{\corrector_T v}{\Omega} \normHsemi{v}{T}
        + \|b\|_{L^\infty(T)}\normHsemi{v}{T} \normL{\corrector_T v}{\Omega}\\
   &\leq \left(\epsilon + H\|b\|_{L^\infty(T)}\CI\right) 
      \normHsemi{\corrector_{T} v}{\Omega}
        \normHsemi{v}{T},
\end{align*}
where the definition of the element corrector in~\eqref{eqn:corrector1_local}
implies the second equality and the approximation property~\eqref{eqn:interpolation}
leads to the last inequality.
This proves the assertion.
\end{proof}


The following theorem assembles the local estimates from 
Lemma~\ref{lemma:LocalGlobal} to derive an estimate for the global 
corrector.

\begin{theorem}\label{theorem:decay}
Given $v\in V_H$ and $\ell\in\mathbb{N}_{+}$, it holds that
\begin{align}\label{eqn:GLAppro}
\normHsemi{\corrector v-\corrector_{\ell}v}{\Omega}
     \lesssim  C(H,h,\epsilon,b,\ell)  \beta^{\ell-1}
   \normHsemi{v}{\Omega}
\end{align}
with
\begin{equation}\label{eqn:defconstLoc}
\begin{aligned}
  C(H,h,\epsilon,b,\ell)&:= 
     \left( 1+\Peclet{H}\CI\right)^2 
        (\CI+1) \\
    &\qquad \times
        \left( 1+2\CI + \Peclet{H} \CI\right)
        \Col[\ell+2]^{1/2}.
\end{aligned}
\end{equation}
\end{theorem}

\begin{proof}
Set $z:=\corrector v-\corrector_{\ell-2}v\in \mathrm{Ker}I_H$ and 
$z_T:=\corrector_{T}v-\corrector_{T,\ell-2}v$, then $z=\sum\limits_{T\in\tri_H}z_T$. We have
\begin{align}\label{eqn:454}
\epsilon\normHsemi{z}{\Omega}^2=\sum\limits_{T\in\tri_H}a(z,z_T).
\end{align}
We estimate $a(z,z_T)$ for each coarse element $T\in \mathcal{T}_H$. 
Recall that we defined a cutoff function $\eta$ in the proof of 
Theorem~\ref{thm:decay}. Note that $\Omega_{T,\ell-2,b}\subset S_{T,\ell-1,b}$.
By construction,  we have 
$\eta z\in R_H(\Omega\setminus S_{T,\ell-1,b})\subset R_H(\Omega\setminus \Omega_{T,\ell-2,b})$. 
Since
$\corrector_{T,\ell-2} v\vert_{\Omega\setminus\Omega_{T,\ell-2,b}}=0$,
this implies
\begin{align*}
 a(\eta z,z_T)=a(\eta z,\corrector_{T}v).
\end{align*}
Furthermore, notice that $\eta z\in \mathrm{ker}(I_H)$, which combined 
with~\eqref{eqn:defloccor}
yields
\begin{align*}
 a(\eta z,\corrector_{T}v)=a_{T}(\eta z,v)=0.
\end{align*}
As a consequence, we obtain
\begin{align*}
a(z,z_T)=a(\eta z,z_T)+a((1-\eta)z,z_T)=a((1-\eta)z,z_T).
\end{align*}
In the following, we will bound the term $a((1-\eta)z,z_T)$. 
Recall from the proof of Theorem~\ref{thm:decay} 
that $(1-\eta)\vert_{\Omega\setminus S_{T,\ell,b}}=0$, 
$\|\nabla(1-\eta)\|_{L^\infty(\Omega)}\leq 2H^{-1}$ and 
$\|(1-\eta)\|_{L^\infty(\Omega)}\leq 1$.
Taking into account that $I_H(z)=I_H(z_T)=0$, the stability 
of the projector $I_H$ from~\eqref{eqn:interpolation}, therefore,  
leads to
\begin{align*}
  &a((1-\eta)z,z_T)\\
   &\qquad\leq 
    \epsilon \normHsemi{(1-\eta) z}{S_{T,\ell,b}} \normHsemi{z_T}{S_{T,\ell,b}}
      + \|b\|_{L^\infty(\Omega)} \normL{z}{S_{T,\ell,b}} \normHsemi{z_T}{S_{T,\ell,b}}\\
     &\qquad\leq \left(\epsilon (1+2\CI) + \|b\|_{L^\infty(\Omega)} H \CI\right)
        \normHsemi{z}{S_{T,\ell,b}} \normHsemi{z_T}{S_{T,\ell,b}}.
\end{align*}
Since $S_{T,\ell,b}\subset \Omega_{T,\ell,b}$, the combination 
with~\eqref{eqn:454} and the application of a discrete Cauchy-Schwarz 
inequality yields
\begin{align*}
  \normHsemi{z}{\Omega}^2
    &\leq \left( 1+2\CI + \Peclet{H} \CI\right)
      \sum\limits_{T\in\tri_H}\normHsemi{z}{S_{T,\ell,b}}
          \normHsemi{z_T}{S_{T,\ell,b}}\\
   &\leq \left( 1+2\CI + \Peclet{H} \CI\right)\\
   & \qquad\qquad\qquad \times
     \Bigg(\sum\limits_{T\in\tri_H} \normHsemi{z}{\Omega_{T,\ell,b}}^2\Bigg)^{1/2}
     \Bigg(\sum\limits_{T\in\tri_H}\normHsemi{z_T}{S_{T,\ell,b}}^2\Bigg)^{1/ 2}.
\end{align*}
Lemma~\ref{lemma:LocalGlobal} implies 
\begin{align*}
 &\Bigg(\sum\limits_{T\in\tri_H}\normHsemi{z_T}{S_{T,\ell,b}}^2\Bigg)^{1/ 2}\\
  &\qquad\qquad  \lesssim \left( 1+\Peclet{H}\CI\right)^2 
        (\CI+1)  \beta^{\ell-3}
   \normHsemi{v}{\Omega},
\end{align*}
while the bounded overlap of the patches from~\eqref{eqn:overlap} implies 
\begin{align*}
 \Bigg(\sum\limits_{T\in\tri_H} \normHsemi{z}{\Omega_{T,\ell,b}}^2\Bigg)^{1/2}
   \leq \Col[\ell]^{1/2} \normHsemi{z}{\Omega}.
\end{align*}
In the end, the combination of the previous displayed inequalities and 
the shift $\ell\mapsto \ell+2$ shows the assertion.
\end{proof}

Now we are ready to define the localized multiscale test space as
\begin{align*}
  W_{H,\ell}=(1-\corrector_{\ell} )V_H.
\end{align*}
The Petrov-Galerkin method for the approximation of \eqref{eqn:fine-scale} based on 
the trial-test pairing $(V_{H},W_{H,\ell})$ defined above seeks $u_{H,\ell}\in V_{H}$ satisfying
\begin{align}\label{eqn:LOD}
  a(u_{H,\ell},w_{H,\ell})= \langle f,w_{H,\ell}\rangle_{H^{-1}(\Omega)\times H^1_0(\Omega)}
  \qquad \text{for all }w_{H,\ell}\in W_{H,\ell}.
\end{align}

\begin{lemma}[Inf-sup stability]\label{lemma:stabilityLOD}
If $\ell$ is sufficiently large, i.e., the oversampling condition 
\begin{align}\label{eqn:ell}
  \ell\gtrsim 
    \frac{1+\lvert\log(\CI)\rvert+\lvert\log(C(H,h,\epsilon,b,\ell)\rvert
         +\lvert\log(1+\CI+\Peclet{H} \CI^2)\rvert}{\lvert
            \log(4\CI+3\CI^2)-\log(1+4\CI+3\CI^2)\rvert}
\end{align}
is satisfied,
then the Petrov-Galerkin method \eqref{eqn:LOD} is inf-sup stable and 
\begin{align}\label{eqn:infsupLOD}
\inf\limits_{w_{H,\ell}\in W_{H,\ell}\setminus\{0\}}\sup\limits_{u_{H}\in V_{H}\setminus\{0\}}
          \frac{a(u_{H},w_{H,\ell} )}{\normHsemi{u_{H}}{\Omega}\normHsemi{w_{H,\ell}}{\Omega}}
            \gtrsim  \frac{\epsilon}{\CI}.
\end{align}
\end{lemma}

\begin{remark}
If $H/h=1/\epsilon$ and $\lvert b\rvert=1$, then~\eqref{eqn:ell} reads 
\begin{align*}
 \ell\gtrsim (\log(\epsilon))^2,
\end{align*}
i.e., the local patch size $\ell$ depends on $\log(\epsilon)$ algebraically. 
\end{remark}

\begin{remark}
Since the dimension of $V_H$ equals the dimension of $W_{H,\ell}$, the 
reverse inf-sup condition 
\begin{align}\label{eqn:infsupLOD2}
\inf_{u_{H}\in V_{H}\setminus\{0\}}\sup_{w_{H,\ell}\in W_{H,\ell}\setminus\{0\}}
          \frac{a(u_{H},w_{H,\ell} )}{\normHsemi{u_{H}}{\Omega}\normHsemi{w_{H,\ell}}{\Omega}}
            \gtrsim  \frac{\epsilon}{\CI}.
\end{align}
follows from Lemma~\ref{lemma:stabilityLOD}.
\end{remark}

\begin{proof}[Proof of Lemma~\ref{lemma:stabilityLOD}]
Let $w_{H,\ell}\in W_{H,\ell}$, and set $w_{H}=(1-\corrector )I_{H}w_{H,\ell}\in W_H$. 
By Lemma \ref{lemma:infsup}, there exists $u_H\in V_H$, s.t.,
\begin{align}\label{eqn:777}
  a(u_{H},w_{H} )
  \geq \frac{\epsilon}{\CI}\normHsemi{u_{H}}{\Omega}\normHsemi{w_{H}}{\Omega}.
\end{align}

Taking into account that $w_{H,\ell}=I_H w_{H,\ell}-\corrector_\ell I_H w_{H,\ell}$, 
we arrive at $w_H-w_{H,\ell} = (\corrector_\ell-\corrector)I_H w_{H,\ell}$. 
As a consequence, 
Theorem~\ref{theorem:decay} together with the stability of $I_H$ 
from~\eqref{eqn:interpolation} implies 
\begin{align*}
 \normHsemi{w_{H}-w_{H,\ell}}{\Omega} 
    &\leq \tilde{C} C(H,h,\epsilon,b,\ell) \beta^{\ell-1}
            \normHsemi{I_H w_{H,\ell}}{\Omega}\\
    &\leq \tilde{C} C(H,h,\epsilon,b,\ell) \CI \beta^{\ell-1}
            \normHsemi{w_{H,\ell}}{\Omega}.
\end{align*}
Here, $\tilde{C}$ denotes the constant hidden in $\lesssim$ in Theorem~\ref{theorem:decay}, 
which is independent of $H$, $h$ or $\epsilon$. 
The combination with a triangle inequality leads to
\begin{align*}
 \normHsemi{w_H}{\Omega}
   &\geq \normHsemi{w_{H,\ell}}{\Omega} - \normHsemi{w_h-w_{H,\ell}}{\Omega}\\
   &\geq (1-\tilde{C}C(H,h,\epsilon,b,\ell) \CI \beta^{\ell-1})\normHsemi{w_{H,\ell}}{\Omega}.
\end{align*}
Since $I_H(w_{H,\ell}-w_H)=0$, i.e., $w_{H,\ell}-w_H\in R_H$, this leads to 
\begin{align*}
 \lvert a(u_{H},w_{H,\ell}-w_{H})\rvert
   &\leq (\epsilon + \|b\|_{L^\infty(\Omega)} H \CI) 
      \normHsemi{u_H}{\Omega} \normHsemi{w_{H,\ell}-w_H}{\Omega}\\
&=\epsilon (1+\text{Pe}_{H,b,\epsilon}\CI) \normHsemi{u_H}{\Omega} \normHsemi{w_{H,\ell}-w_H}{\Omega}.
\end{align*}
The combination of the above displayed inequalities results in 
\begin{align*}
  &a(u_{H},w_{H,\ell})=a(u_{H},w_{H})+a(u_{H},w_{H,\ell}-w_{H})\\
  &\geq \frac{\epsilon}{\CI}(1-\tilde{C}C(H,h,\epsilon,b,\ell) \CI \beta^{\ell-1}) 
           \normHsemi{u_{H}}{\Omega}\normHsemi{w_{H,\ell}}{\Omega}\\
  &\qquad
   - \epsilon(1 + \Peclet{H}\CI) C(H,h,\epsilon,b,\ell) \tilde{C} \CI \beta^{\ell-1} 
    \normHsemi{u_{H}}{\Omega}\normHsemi{w_{H,\ell}}{\Omega}.
\end{align*}
Recall the definition of $\beta$ from~\eqref{eqn:defbeta}.
If $\ell$ satisfies \eqref{eqn:ell}, then we obtain~\eqref{eqn:infsupLOD}.
\end{proof}

We are ready to estimate the error $\normHsemi{u_{H}-u_{H,\ell}}{\Omega}$
coming from the localization.
\begin{lemma}\label{lemma:discrerrorlocal}
Let $\ell$ satisfy~\eqref{eqn:ell}. Then
\begin{equation}\label{eqn:errorGL}
\begin{aligned}
  \normHsemi{u_{H}-u_{H,\ell}}{\Omega}
    &\lesssim \CI^2 C(H,h,\epsilon,b,\ell) (1+\Peclet{H}\CI)\\
      &\qquad\qquad\qquad\qquad\qquad \times \beta^{\ell-1}
           \normHsemi{u_h-u_H}{\Omega}.
\end{aligned}
\end{equation}
\end{lemma}

\begin{proof}
Notice that $u_{H}-u_{H,\ell}\in V_H$ is a coarse finite element function. 
Therefore, the 
inf-sup condition~\eqref{eqn:infsupLOD2} guarantees the existence of 
$w_{H,\ell}\in W_{H,\ell}$ with 
\[\normHsemi{u_{H}-u_{H,\ell}}{\Omega}
\lesssim \frac{\CI}{\epsilon}\frac{a(u_{H}-u_{H,\ell},w_{H,\ell})}{\normHsemi{w_{H,\ell}}{\Omega}}.\]
In view of $w_{H,\ell}\in W_{H,\ell}\subset V_h$, the standard Galerkin problem~\eqref{eqn:fine-scale}
and the VMS~\eqref{eqn:LOD} imply 
\begin{align*}
  a(u_{H}-u_{H,\ell},w_{H,\ell}) = a(u_H-u_h,w_{H,\ell}).
\end{align*}
Define $w_{H}:=I_H w_{H,\ell} - \corrector(I_{H}w_{H,\ell})\in W_{H}\subset V_h$.
Together with the orthogonality of Petrov-Galerkin type,  we obtain 
\begin{align*}
  a(u_H-u_h,w_{H,\ell}) = a(u_H-u_h,w_{H,\ell}-w_H).
\end{align*}
Taking into account that $w_{H,\ell}-w_H= \corrector_\ell I_H w_{H,\ell} - \corrector I_H w_{H,\ell}$, 
the combination with a Cauchy inequality, $w_{H,\ell}-w_H\in\mathrm{ker}(I_H)$ 
and an application of Theorem~\ref{theorem:decay} lead to
\begin{align*}
  &a(u_H-u_h,w_{H,\ell}-w_H)\\
  &\qquad
     \lesssim (\epsilon+ \|b\|_{L^\infty(\Omega)}H\CI) C(H,h,\epsilon,b,\ell) \beta^{\ell-1}
         \normHsemi{I_H w_{H,\ell}}{\Omega} \normHsemi{u_h-u_H}{\Omega}.
\end{align*}
The stability of $I_H$ from~\eqref{eqn:interpolation} implies the assertion.
\end{proof}

Lemma~\ref{lemma:discrerrorlocal} allows bounding the error for 
the localized VMS in the following manner.

\begin{theorem}[global error estimate for localized VMS]\label{thm:globalerrorboundVMS}
Let $\ell$ satisfy~\eqref{eqn:ell},
then 
\begin{align*}
  \normHsemi{u_h-u_{H,\ell}}{\Omega}
    &\lesssim \left(\CI + \CI^3 C(H,h,\epsilon,b,\ell)
         (1+\Peclet{H}\CI)\beta^{\ell-1}\right)\\
    &\qquad\qquad\qquad\qquad\qquad \qquad\qquad\qquad
       \times\min_{v_H\in V_H} \normHsemi{u_h-v_H}{\Omega}
\end{align*}
with the constant $(C(H,h,\epsilon,b,\ell)$ from~\eqref{eqn:defconstLoc}.
\end{theorem}

\begin{proof}
The proof follows directly from a triangle inequality, Proposition~\ref{prop:ideal},
and Lemma~\ref{lemma:discrerrorlocal}.
\end{proof}

Although Theorem~\ref{thm:globalerrorboundVMS} provides a best-approximation 
result, the assertion still depends on $\epsilon$, which is hidden in the 
best-approximation $\min_{v_H\in V_H} \normHsemi{u_h-v_H}{\Omega}$.
The locality in the error bound of the ideal method from Proposition~\ref{prop:ideal}
transfers to the VMS defined in~\eqref{eqn:LOD} and results in the 
local error bound in the following theorem. 
Note that the error from the localization still depends on the mesh P\'eclet
number of $\tri_H$ and still contains the best-approximation error on the 
whole domain. 
Nevertheless, this ill-behaved terms are weighted by the exponentially 
decaying term $\beta^{\ell-1}$, where $\beta$ is bounded above from 1. 

\begin{theorem}[local error estimate for localized VMS]\label{thm:locerrorVMS}
Let $\ell$ satisfy~\eqref{eqn:ell}. 
Then for any $\mathcal{K}\subset\tri$ and $\omega:=\cup\mathcal{K}$, it 
holds that
\begin{align*}
  &\normHsemi{u_h-u_{H,\ell}}{\omega}
    \lesssim \CI \min_{v_H\in V_H} \normHsemi{u_h-v_H}{\omega} \\
     &\qquad\quad+ \CI^3 C(H,h,\epsilon,b,\ell)(1+\Peclet{H}\CI)\beta^{\ell-1}
             \min_{v_H\in V_H} \normHsemi{u_h-v_H}{\Omega}.
\end{align*}
\end{theorem}

\begin{remark}[complexity]
The problem~\eqref{eqn:LOD} on the coarse scale consists of $\mathcal{O}(1/H^2)$
degrees of freedom (DOFs). Corresponding to each of those DOFs, one localized 
corrector problem~\eqref{eqn:corrector1_local} has to be solved, which relates to 
$\mathcal{O}(\ell^2 H^3/ (h^2\epsilon))$ DOFs in the worst case scenario. 
If the mesh is structured, the number of corrector problems that have 
to be solved can be reduced to $\mathcal{O}(\ell H/\epsilon)$, 
cf.\ \cite{Gallistl.Peterseim:2015}.
\end{remark}

\section{Numerical experiment}\label{s:numerics}
In this section, we present one simple numerical test to illustrate the
theoretical convergence results of the 
localized method proposed in \eqref{eqn:LOD}. 
We take $\Omega=(0,1)\times(0,1)$, the velocity field
$b=(\cos(0.7),  \sin(0.7))^\top$, the volume force $f\equiv 1$ and 
$\epsilon = 2^{-7}$.
The reference solution $u_h$ is obtained through \eqref{eqn:fine-scale} by 
taking $h=\sqrt{2}\,2^{-8}$. 

We will compare our approache with SUPG. Let us briefly review
the SUPG model to \eqref{eq:original} \cite{FRANCA1992253}.
Let $(\bullet,\bullet)_T:=(\bullet,\bullet)_{L^2(T)}$ denote the 
$L^2$ scalar product over a triangle $T\in\tri_H$.
Then SUPG seeks $u_H\in V_H$ such that 
\begin{align}\label{eq:SUPG}
B_{\mathrm{SUPG}}(u_H^{\mathrm{SUPG}},v_H)=F_{\mathrm{SUPG}}(v_H) 
  \qquad \text{for all }v_H\in V_H
\end{align}
with 
\begin{align*}
B_{\mathrm{SUPG}}(u_H^{\mathrm{SUPG}},v_H)=a(u_H^{\mathrm{SUPG}},v_H)
    +\delta_\mathrm{SUPG}\sum\limits_{T\in \mathcal{T}_H} 
         (b\cdot\nabla u_H^{\mathrm{SUPG}},b\cdot\nabla v_H)_{T}
\end{align*}
and 
\begin{align*}
F_{\mathrm{SUPG}}(v_H)
   =\langle f,v_H\rangle_{H^{-1}(\Omega)\times H^1_0(\Omega)}
 +\delta_\mathrm{SUPG}\sum\limits_{k\in \mathcal{T}_H} (f,b\cdot\nabla v_H)_{T}.
\end{align*}
Here,  $\delta_\mathrm{SUPG}$ indicates the stability parameter, and we 
choose
 \[
\delta_\mathrm{SUPG} = \frac{H}{\sqrt{8}\max(\epsilon,H/\sqrt{2})}
\]
in our numerical test. 

The reference solution from \eqref{eqn:fine-scale} and the coarse scale solution 
from \eqref{eqn:LOD} and the SUPG solution from~\eqref{eq:SUPG} 
with $H=\sqrt{2}\,2^{-4}$ are depicted in 
Figure~\ref{fig:solutions}. 
One can observe that the classical FEM approximation with $H=\sqrt{2}\,2^{-4}$ is not 
stable around the boundary layers (i.e. the top and right boundaries)
and shows spurious oscillations, and thus 
fails to provide a reliable solution.
Nevertheless, both the SUPG method and the ideal method are stable and generate 
an accurate solution. We display the solutions for fixed $y=0.75$ to illustrate 
the stability and accuracy of the VMS method in Figure~\ref{fig:solutions_y}. 
We observe large oscillations in the coarse scale solution obtained 
through classical FEM when $x$ approaches 1, while the SUPG and the 
VMS method yield reliable solutions. The smearing is restricted to one layer 
of elements around the boundary.
We can also conclude that the SUPG and 
the VMS method reproduce the reference solution away from $x=1$ and the 
latter shows slightly less smearing. We want to 
highlight that the localization parameter is $\ell=1$ for the VMS method 
in this example.

\begin{figure}
  \begin{center}
     \includegraphics[width=0.49\textwidth]{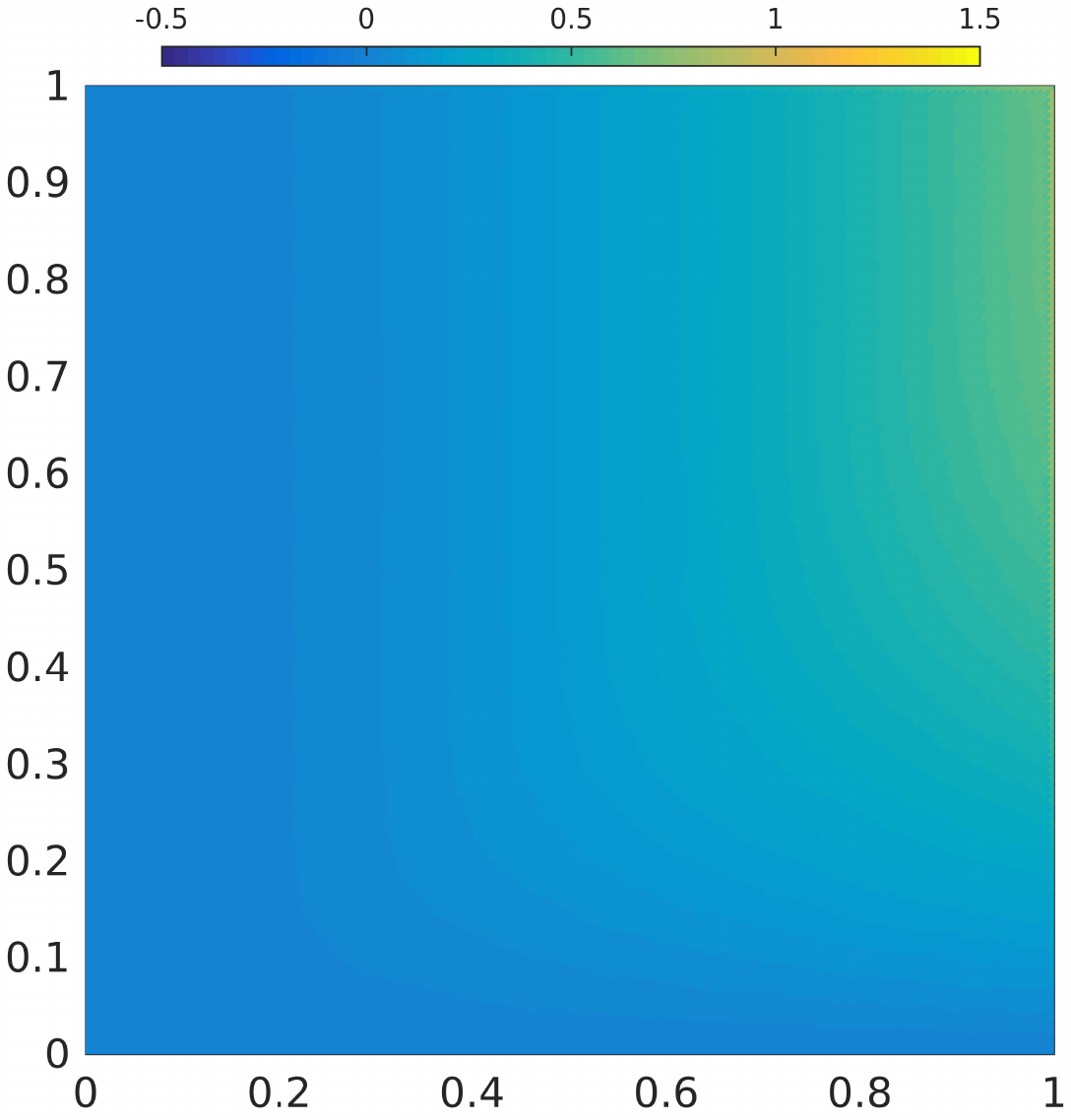}
     \includegraphics[width=0.49\textwidth]{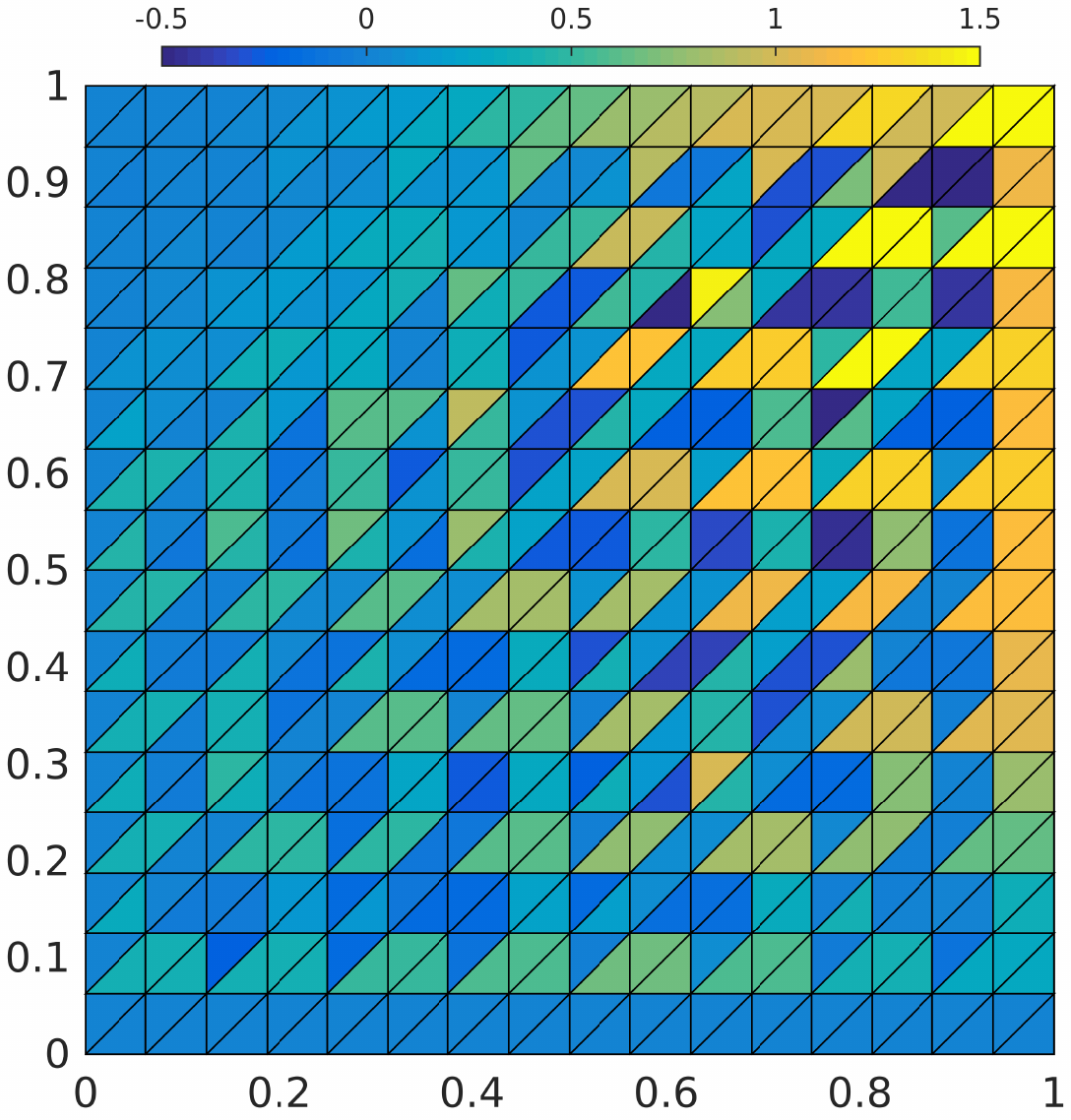}\\
     \includegraphics[width=0.49\textwidth]{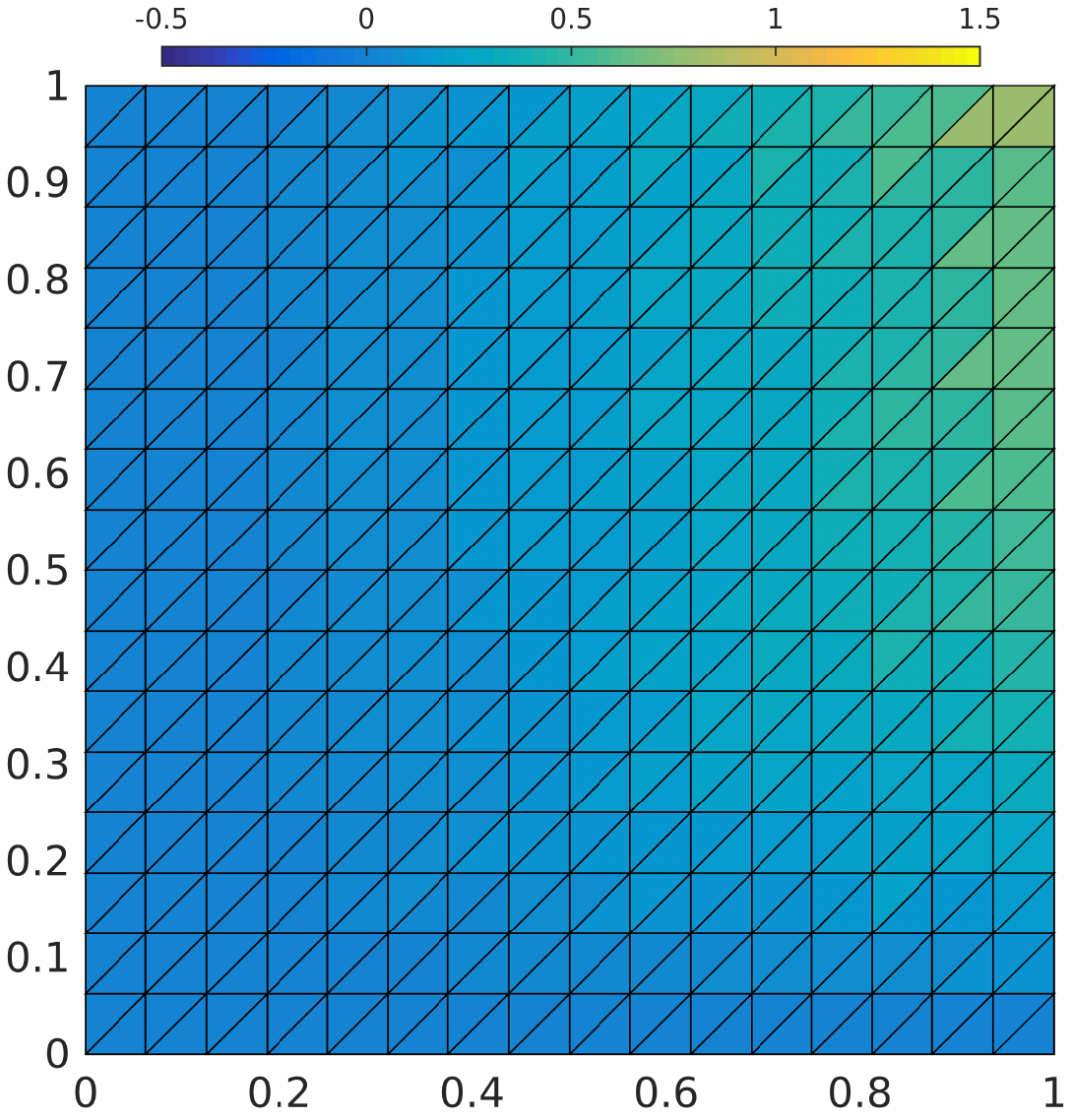}
     \includegraphics[width=0.49\textwidth]{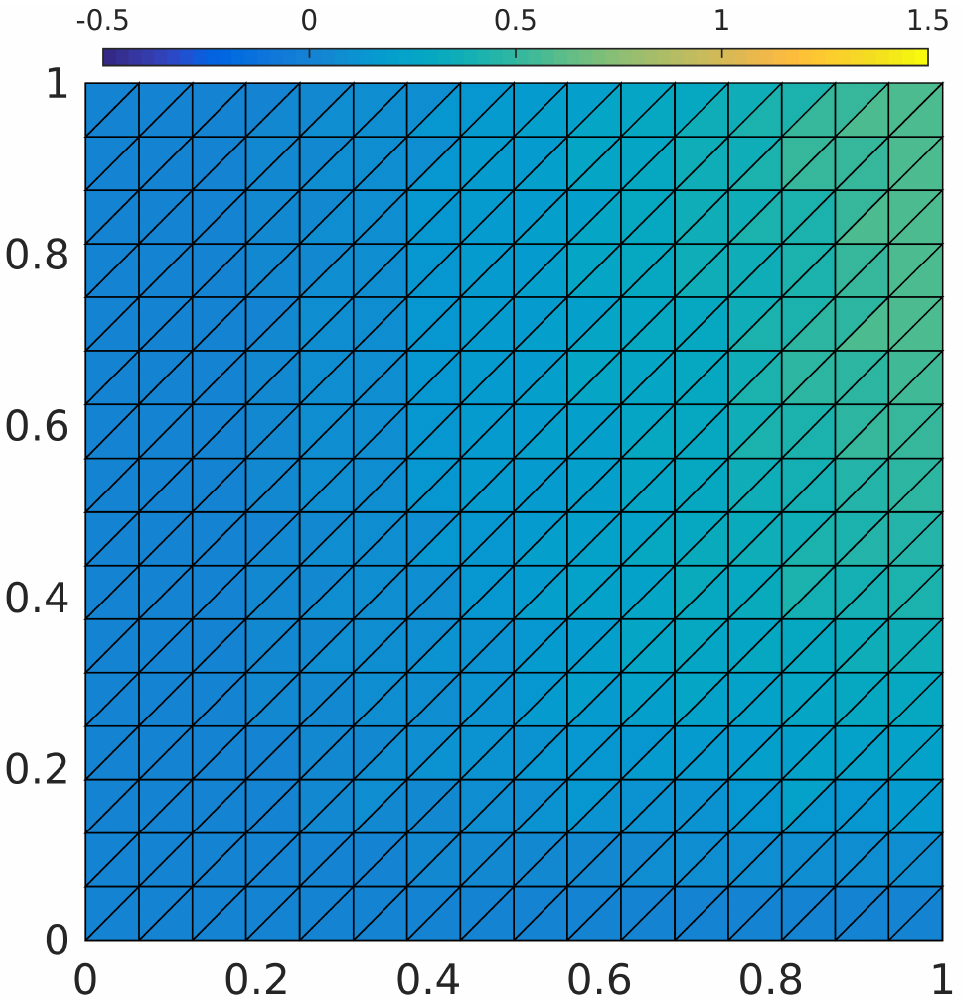}
  \end{center}
  \caption{Reference solution (top left), classical FEM approximation (top right),
          SUPG approximation (bottom left), and multiscale approximation for $\ell=1$ (bottom right)
          for $\epsilon=2^{-7}$ and $H=\sqrt{2}\,2^{-4}$.}
\label{fig:solutions}
\end{figure}

\begin{figure}
  \begin{center}
    \includegraphics[width=0.7\textwidth]{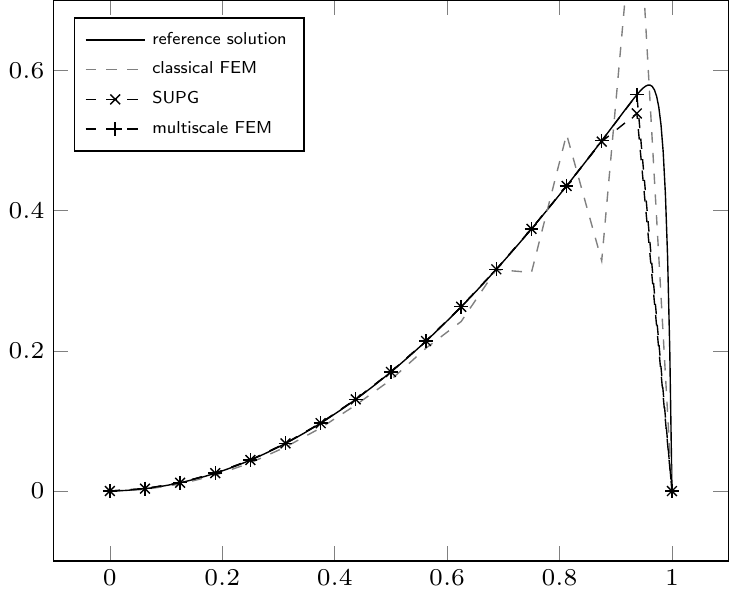}
  \end{center}
  \caption{Reference solution, classical FEM, SUPG and multiscale approximation for $\ell=1$
       at the line $y=0.75$ on a mesh with mesh-size $H=\sqrt{2}\,2^{-4}$
       for $\epsilon=2^{-7}$.}
\label{fig:solutions_y}
\end{figure}

Tables~\ref{table:h1_layer} and \ref{table:l2_layer} display 
the errors between the localized 
solutions \eqref{eqn:LOD} and the reference solution $u_h$
under various coarse mesh-sizes $H$ and localization parameters $\ell$. 
We observe an optimal convergence rate of $\mathcal{O}(H)$ in 
Table~\ref{table:h1_layer} for the error in the $H^1$ semi norm in the domain 
$[0,0.75]\times [0,0.75]$ away from the boundary layers, and  
an optimal convergence rate of $\mathcal{O}(H^2)$ in Table~\ref{table:l2_layer} 
for the global error in the $L^2$ norm. Although 
Theorem~\ref{thm:locerrorVMS} guarantees optimality only under the 
assumption that $\ell$ is large enough in the sense of~\eqref{eqn:ell},
the numerical experiment demonstrates that $\ell=1$ is sufficient 
for an accurate solution, which implies a huge potentially computational reduction. 

The convergence rate for $u_{H,1}$ with various $\epsilon$ in a range 
from $2^{-5}$ to $2^{-8}$ is shown in Figures~\ref{fig:h1_local_convergence} 
and \ref{fig:l2_global_convergence}. 
The error is stable and of order $\mathcal{O}(H)$ 
with respect to the $H^1$ semi norm in a region away from boundary layers 
and of order $\mathcal{O}(H^2)$ in the global $L^2$ norm with a 
preasymptotic effect for smaller values of $\epsilon$. For comparison, the nodal 
interpolation error (i.e. the error from the ideal method) in the global $L^2$ 
norm is depicted, which agrees with $\norm{u_h-u_{H,1}}_{L^2(\Omega)}$ very well. 
This justifies the fast convergence of the localized method with respect 
to the localization parameter $\ell$ for all of the considered values 
of $\epsilon$. 

\begin{table}[htp]
\begin{center}
\begin{tabular}{l|llllll}
 &$\ell=1$&	$\ell=2$&	$\ell=3$&	$\ell=4$&	$\ell=5$&	$\ell=6$\\
\hline
$H=\sqrt{2}2^{-3}$ 	&5.14e-02	&5.14e-02	&5.14e-02	&5.14e-02	&5.14e-02	&5.14e-02	\\ 
$H=\sqrt{2}2^{-4}$	&2.57e-02	&2.57e-02	&2.57e-02	&2.57e-02	&2.57e-02	&2.57e-02	\\ 
$H=\sqrt{2}2^{-5}$	&1.27e-02	&1.27e-02	&1.27e-02	&1.27e-02	&1.27e-02	&1.27e-02	\\ 
$H=\sqrt{2}2^{-6}$	&6.23e-03	&6.23e-03	&6.23e-03	&6.23e-03	&6.23e-03	&6.23e-03
\end{tabular}
\end{center}
\caption{The error $\|\nabla(u_h-u_{H,\ell})\|_{L^2(\Omega_r)}$ for 
 $\Omega_r=[0,0.75]\times [0,0.75]$ for different localization parameters $\ell$
 and mesh-sizes $H$ for $\epsilon=2^{-7}$.}
\label{table:h1_layer}
\end{table}

\begin{table}[htp]
\begin{center}
\begin{tabular}{l|llllll}
 &$\ell=1$&	$\ell=2$&	$\ell=3$&	$\ell=4$&	$\ell=5$&	$\ell=6$\\
\hline
$H=0.17678$ 	&9.45e-02	&9.45e-02	&9.45e-02	&9.45e-02	&9.45e-02	&9.45e-02	\\ 
$H=0.088388$	&5.34e-02	&5.34e-02	&5.34e-02	&5.34e-02	&5.34e-02	&5.34e-02	\\ 
$H=0.044194$	&2.31e-02	&2.32e-02	&2.32e-02	&2.32e-02	&2.32e-02	&2.32e-02	\\ 
$H=0.022097$	&7.25e-03	&7.27e-03	&7.27e-03	&7.27e-03	&7.27e-03	&7.27e-03	\\ 
\end{tabular}
\end{center}
\caption{The error $\|u_h-u_{H,\ell}\|_{L^2(\Omega)}$ for different 
 localization parameters $\ell$  and mesh-sizes $H$ for $\epsilon=2^{-7}$.}
\label{table:l2_layer}
\end{table}

\begin{figure}
  \begin{center}
    \includegraphics[width=0.7\textwidth]{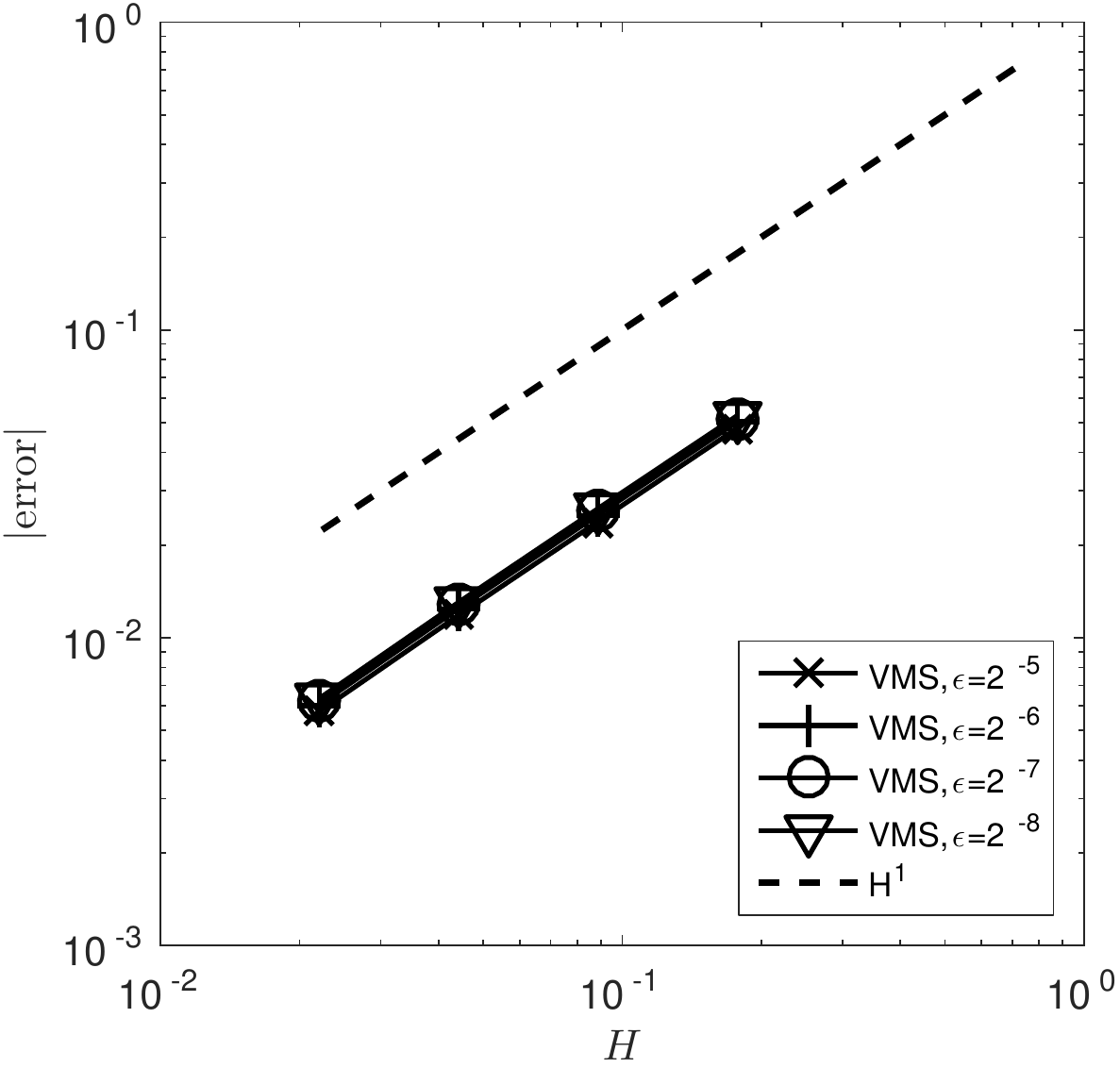}
  \end{center}
  \caption{The errors $\|\nabla(u_h-u_{H,1})\|_{L^2(\Omega_r)}$ for 
  $\Omega_r=[0,0.75]\times [0,0.75]$ for different values of $\epsilon$.}
\label{fig:h1_local_convergence}
\end{figure}

\begin{figure}
  \begin{center}
    \includegraphics[width=0.7\textwidth]{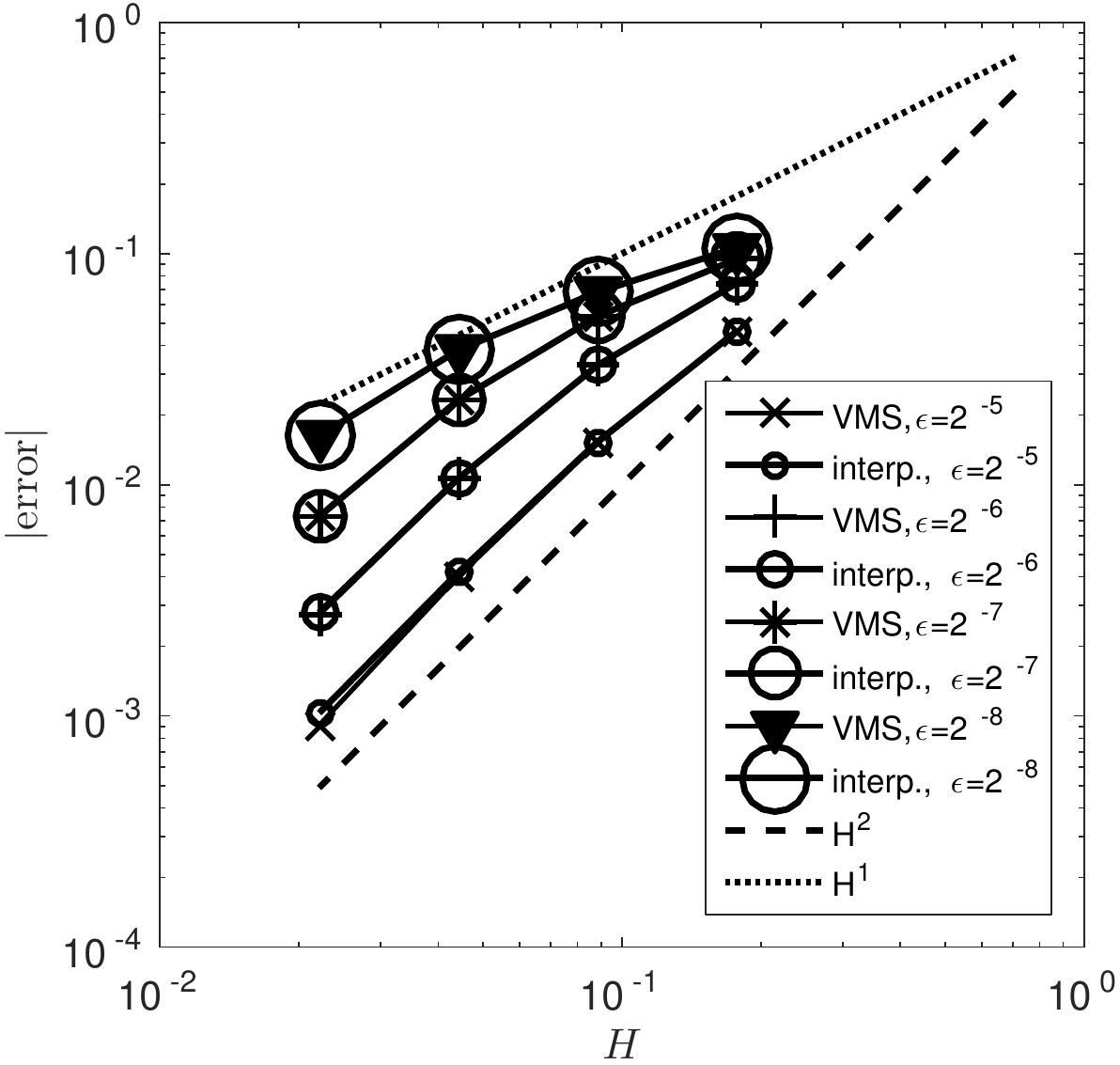}
  \end{center}
  \caption{The errors $\|u_h-u_{H,1}\|_{L^2(\Omega)}$ and 
       $\|u_h-I_H u_h\|_{L^2(\Omega)}$ for different values of $\epsilon$.}
\label{fig:l2_global_convergence}
\end{figure}

\section{Conclusions}\label{s:conclusions}
In this paper, a singularly perturbed convection-diffusion equation was 
considered, and we obtained a stable locally quasi-optimal variational multiscale 
method based on the nodal interpolation operator. 
Due to the high complexity involved in solving the global correctors, 
which account for the main component of the variational multiscale method, 
a further model reduction was proceeded  
by localization techniques based on the LOD method. 
This localization employs local patches which depend on the velocity 
field $b$ and the singular perturbation parameter $\epsilon$.
The error of the localization decays exponentially.
We also provided a numerical experiment to illustrate our theoretical results.  

The stability constant of the nodal interpolation operator that occurs 
in the error estimate, depends logarithmically on $H/h$ (and so on $\epsilon$).
In the three-dimensional case, this stability estimate depends polynomially 
on $H/h$. Therefore a generalization of the proposed method to 3D seems 
to be not reasonable.

The local patches in the localized computation of the corrector depend on $\epsilon$. 
It is an open question, if this is optimal or if a further reduction or simplification 
is possible.

\bibliographystyle{abbrv}
\bibliography{reference}



\end{document}